\documentclass[12pt,a4paper,psamsfonts]{amsart}
\usepackage{fancyhdr}
\usepackage{appendix}
\usepackage{amssymb,amscd,amsxtra,calc}
\usepackage{mathrsfs}
\usepackage{amsmath}
\usepackage{cmmib57}
\usepackage{multirow}
\usepackage{verbatim}
\usepackage[all]{xy}
\usepackage{url}
\usepackage[colorlinks,linkcolor=blue,anchorcolor=blue,citecolor=green]{hyperref}
\theoremstyle{plain}
    \newtheorem{thm}{Theorem}[section]

     \newtheorem{defi}[thm]{Definition}
    \newtheorem{coro}[thm]{Corollary}
    
    \newtheorem{lem}[thm]{Lemma}
    \newtheorem{pro}[thm]{Proposition}
    \newtheorem{question}[thm]{Question}
    
    \newtheorem{remark}[thm]{Remark}

\makeatletter

\newcommand{\Rmnum}[1]{\expandafter\@slowromancap\romannumeral #1@}
\makeatother

\begin{document}

\bibliographystyle{plain}
\title[Int-amplified endomorphisms of compact k\"ahler  spaces]
{Int-amplified endomorphisms of compact k\"ahler spaces
}

\author{Guolei Zhong}
\address
{
%\textsc{Department of Mathematics} \endgraf
\textsc{National University of Singapore,
Singapore 119076, Republic of Singapore
}}
\address{
\textsc{Current address: Center for Complex Geometry, Institute for Basic Science, 55, Expo-ro, Yuseong-gu, Daejeon, Republic of Korea, 34126
}}
\email{zhongguolei@u.nus.edu, guolei@ibs.re.kr}

\begin{abstract}
Let $X$ be a normal compact K\"ahler space of dimension $n$. A  surjective endomorphism $f$ of such $X$ is int-amplified if $f^*\xi-\xi=\eta$ for some K\"ahler classes $\xi$ and $\eta$. First, we show that this definition generalizes the notion in the projective setting. Second, we prove that for the cases of $X$ being smooth, a surface or a threefold with mild singularities,  if $X$ admits an int-amplified endomorphism with pseudo-effective canonical divisor, then it is a $Q$-torus. Finally, we consider a  normal compact K\"ahler threefold $Y$ with only terminal singularities and show that, replacing $f$ by a positive power, we can run the minimal model program (MMP) $f$-equivariantly for such $Y$ and reach either a $Q$-torus or a Fano (projective) variety of Picard number one.
\end{abstract}


\maketitle

\tableofcontents

\section{Introduction}
We work over the field $\mathbb{C}$ of complex numbers. By the fundamental work of S. Meng and D.-Q. Zhang (cf.~\cite{meng2017building} and \cite{meng2018building}) during the past several years, we have known the building blocks and characteristic properties of polarized and int-amplified endomorphisms of normal projective varieties. Also, Meng and Zhang show that we can run the minimal model program equivariantly for mildly singular normal projective varieties and finally reach either a $Q$-abelian variety or a Fano variety of Picard number one. As an application, one can use the results to study the totally invariant divisors of polarized or int-amplified endomorphisms (cf.~\cite[Theorem 1.3]{zhang2014invariant} and \cite[Theorem 1.1]{zhong2020invariant}). 

Now, we pose a natural question. 
Does there exist such nice building blocks for compact K\"ahler spaces admitting non-isomorphic surjective endomorphisms?
This is open since the minimal model program for higher-dimensional compact K\"ahler spaces is unknown.  

In this article, we consider an arbitrary normal compact K\"ahler space $X$. A  surjective endomorphism $f$ of $X$ is said to be \textit{int-amplified}, if $f^*\xi-\xi=\eta$ for some K\"ahler classes $\xi$ and $\eta$. We first prove that when $X$ is projective, the generalized definition coincides with the previous one (cf.~\cite{meng2017building}). Then, we follow the idea of \cite{meng2017building} to study the normal compact K\"ahler space admitting an int-amplified endomorphism in terms of the K\"ahler cone and canonical divisor. We show that most of these properties are preserved when the objects are extended to the analytic setting. Moreover, we study the periodic points or totally invariant analytic subvarieties of some special complex spaces such as $Q$-tori, i.e., quasi-\'etale finite quotients of complex tori. Finally, as a consequence of the existence of the minimal model program (MMP) for compact K\"ahler threefolds (cf.~\cite{Horing2015mori} and \cite{horing2016minimal}), we prove that every int-amplified endomorphism of a compact K\"ahler threefold with at worst terminal singularities has an equivariant descending, when running the MMP.

Our work can be divided into two parts. First, we recall notation and terminology for the singular analytic space such as differential forms, currents and invariant cones. We spend some time showing that all of these objects possess nice properties parallel to normal projective cases. Second, we apply these common properties in the preliminary to generalize results in \cite{meng2017building}. Here, we list them below for the convenience of readers.

Theorem \ref{thm1.1} generalizes \cite[Theorem 1.1]{meng2017building}. The proof of the equivalent conditions  is the same as \cite{meng2017building} while the coincidence of these two definitions of int-amplified endomorphisms when $X$ is projective is new (cf.~Remark \ref{reminitial} and the proof of Theorem \ref{thm1.1} in section \ref{section3}).

\begin{thm}\label{thm1.1}
Suppose that $f:X\rightarrow X$ is a surjective endomorphism of a normal compact K\"ahler space $X$ with at worst rational singularities.  Let $\varphi:=f^*|_{\textup{H}_{\textup{BC}}^{1,1}(X)}$, where $\textup{H}_{\textup{BC}}^{1,1}(X)$ denotes the Bott-Chern cohomology space. Then the following are equivalent.
\begin{enumerate}
\item[(1)] The endomorphism $f$ is int-amplified.
\item[(2)] All the eigenvalues of $\varphi$ are of modulus greater than $1$.
\item[(3)] There exists some big $(1,1)$-class $[\theta]$ such that $f^*[\theta]-[\theta]$ is big, i.e., it can be represented by a K\"ahler current $T$.
\item[(4)] If $C$ is a non-empty and non-zero $\varphi$-invariant convex cone in $\textup{H}_{\textup{BC}}^{1,1}(X)$, then
$\emptyset\neq (\varphi-\text{id})^{-1}(C)\subseteq C$.
\end{enumerate}
Furthermore, if $X$ is projective, then the two definitions of int-amplified endomorphisms coincide (cf.~Definition \ref{defipolarized} in the K\"ahler setting and \cite{meng2017building} in the projective setting).
\end{thm}

As an application of Theorem \ref{thm1.1}, we are able to show the following result,  which asserts that the composition of a sufficient number of iterations of int-amplified endomorphisms is still int-amplified. Readers may refer to \cite[Theorem 1.4]{meng2017building} for the projective version.

\begin{thm}\label{thm1.2}
Let $f$ and $g$ be two surjective endomorphisms of a normal compact K\"ahler space $X$ with at worst rational singularities. Suppose that $f$ is int-amplified. Then both $f^i\circ g$ and $g\circ f^i$ are int-amplified for sufficiently large $i\gg 1$.
\end{thm}

We refer the readers to \cite[Sections 4 and 6]{ueno2006classification} for the definitions of Cartier divisor, linear system and Kodaira dimension for a complex analytic variety. Also, the positivity of differential forms and currents will be discussed in the preliminary part. 

The following theorem is important for the equivariant minimal model program of compact K\"ahler threefolds (cf.~\cite[Theorem 2.5]{meng2017building} for the normal projective case).

\begin{thm}\label{thm1.3}
Let $X$ be a normal compact K\"ahler space with at worst rational singularities. 
Suppose that $X$ admits an int-amplified endomorphism $f$ and the canonical divisor $K_X$ is $\mathbb{Q}$-Cartier.
Then, $-K_X$ is pseudo-effective. In particular, the Kodaira dimension $\kappa(X)\le 0$.
\end{thm}

As an application of Theorem \ref{thm1.3} and the main theorem in \cite{graf2019finite}, we get Theorem \ref{thmtorus} below so that when  running the MMP for compact K\"ahler threefolds, we may reduce the case of canonical divisor $K_X$ being pseudo-effective to the $Q$-torus case.

\begin{thm}\label{thmtorus}
Let $X$ be one of the following cases: (i) a compact K\"ahler manifold of any dimension, (ii) a normal compact K\"ahler threefold with at worst canonical singularities, or (iii) a normal $\mathbb{Q}$-factorial compact K\"ahler surface. Suppose that $X$ admits an  int-amplified endomorphism $f:X\rightarrow X$ and the canonical divisor $K_X$ is pseudo-effective. Then $X$ is a $Q$-torus, i.e., a quasi-\'etale finite quotient of a complex torus. 
\end{thm}

We generalize \cite[Theorem 1.10]{meng2017building} to the following equivariant MMP. We refer to \cite[Section 6]{meng2018building} for the technical details.  The author will highlight and compare the differences between projective varieties and compact K\"ahler spaces in Section \ref{section8}. 

\begin{thm}\label{thmmmpthreefold}
Let $f:X\rightarrow X$ be an int-amplified endomorphism of a normal $\mathbb{Q}$-factorial compact K\"ahler threefold with at worst terminal singularities. Then, replacing $f$ by a positive power, there exist a $Q$-torus $Y$, a morphism $X\rightarrow Y$ and an $f$-equivariant minimal model program of $X$ over $Y$:
$X=X_1\dashrightarrow \cdots\dashrightarrow X_i\dashrightarrow\cdots\dashrightarrow X_r=Y$
(i.e., $f=f_1$ descends to $f_i$ on each $X_i$), with every $X_i\dashrightarrow X_{i+1}$ a contraction of a divisorial ray, a flip or a Mori fibre contraction, of a $K_{X_i}$-negative extremal ray, such that:
\begin{enumerate}
\item If $K_X$ is pseudo-effective, then $X=Y$ and it is a $Q$-torus.
\item If $K_X$ is not pseudo-effective, then for each $i$, $f_i$ is int-amplified and $X_i\rightarrow Y$ is an equi-dimensional (holomorphic) morphism with every fibre irreducible and rationally connected. Also, $X_{r-1}\rightarrow X_r=Y$ is a Mori fibre contraction.
\end{enumerate}
\end{thm}

Finally, we pose the question below. If $X$ is assumed to be projective, the answer is affirmative by Fakhruddin's very motivating result (cf.~\cite[Lemma 5.1]{fakhruddin2002questions}).

\begin{question}\label{question1.6}
Suppose that $f:X\rightarrow X$ is an int-amplified endomorphism of a normal compact K\"ahler space. Does there exist a periodic point of $f$? If so, is the set of $f$-periodic points $\textup{Per}(f)$ open dense in $X$?
\end{question}

\begin{remark}[Differences with previous papers]
\textup{In this article, we fix our attention in the complex analytic setting. Our first work is to sort out the notation and properties in compact K\"ahler spaces which are analogous to normal projective settings. Most of them are well-known to experts, but we were not able to find a good reference for these results. So we give the complete proofs in the preliminary part (cf.~Propositions \ref{projection formula} and \ref{lembignef}). Second, in the non-projective setting, we do not have ample divisors and hence an effective cycle cannot be simply written as a summand of the intersection of ample divisors. However, we shall apply the openness of K\"ahler cones to develop some numerical statements (cf.~Lemmas \ref{lemopenclosed}\,$\sim$\,\ref{lemvanish}). As an application, we apply these new results to show Theorem \ref{thmtorus}. Comparing with \cite[Theorem 5.2]{meng2017building}, we cannot follow the original proof to show Theorem \ref{thmtorus} due to the lack of \cite[Theorem 1.1]{greb2016etale}. }
   %In the article, the author will skip the same proof and highlight the differences when extending previous results.}
\end{remark}

 The proofs of Theorems \ref{thm1.1}\,$\sim$\,\ref{thm1.3} are in Section \ref{section3}. The proof of Theorem \ref{thmtorus} is in Section \ref{section6}. The proof of Theorem \ref{thmmmpthreefold} is in Section \ref{section8}. Now, for the organization of the paper, we begin with  preliminaries in Section 2.

\subsubsection*{\textbf{\textup{Acknowledgments}}}
The author would like to thank Professor De-Qi Zhang for many inspiring ideas and discussions,   Professor Tien-Cuong Dinh for answering   questions about differential forms and currents, and Professor Andreas H\"oring for discussing the finiteness problem of surjective endomorphisms of general compact K\"ahler spaces. He also thanks the referees for very careful reading and many constructive suggestions to improve the paper.
The author was supported by a President's Graduate Scholarship of NUS.

\section{Preliminaries}\label{section2}
Let $X$ be a normal compact complex space. We refer to \cite{ancona2006differential}, \cite{grauert1955zur} and \cite{ueno2006classification} for basic notation and properties for complex spaces.  In this preliminary, we shall recall the differential forms (resp. currents) and their pull-backs (resp. push-forwards) in the singular setting. 

Let $f:X\rightarrow Y$ be a surjective morphism between normal compact complex spaces. Here, a morphism is a holomorphic map in complex geometry. The morphism $f$ is said to be \textit{finite} (resp. \textit{generically finite}) if $f$ is proper and has discrete fibres (resp. proper and finite outside a nowhere dense analytic closed subspace of $Y$).  An \textit{endomorphism} of $X$ is a surjective holomorphic map $f:X\rightarrow X$.

In this article, we mainly deal with the singular case. Most of the properties in complex manifolds need to be extended to the singular setting. The following idea will be used more than once during the preliminary part.

\setlength\parskip{8pt}

\textbf{Main Reduction:} Let $f:X\to Y$ be a surjective morphism between normal compact complex spaces. Since $Y$ is compact, by Hironaka desingularization theorem, there exists a modification $h:\widetilde{Y}\rightarrow Y$ obtained by a finite sequence of blowing-up along smooth subspaces, such that $\widetilde{Y}$ is a complex manifold. By \cite[Chapter \Rmnum{1}, Lemma 7.1]{ancona2006differential}, $h:\widetilde{Y}\rightarrow Y$ lifts to $g':X'\rightarrow X$ obtained by a finite sequence of blowing-up such that $f\circ g'$ factors through $\widetilde{Y}$ with the induced morphism $f':X'\to \widetilde{Y}$. Taking a resolution  $\sigma:\widetilde{X}\rightarrow X'$,  we get $g:=g'\circ\sigma:\widetilde{X}\rightarrow X$, $\widetilde{f}:=f'\circ \sigma:\widetilde{X}\to \widetilde{Y}$ and the commutative diagram $f\circ g=h\circ \widetilde{f}$, such that all of the morphisms are surjective.\setlength\parskip{0pt}
% \[\xymatrix{\widetilde{X}\ar[r]^{\widetilde{f}}\ar[d]_{g}&\widetilde{Y}\ar[d]^{h}\\
%X\ar[r]^{f}&Y.}
%\]
By the classical differential geometry, the morphism $\widetilde{f}$ admits very nice properties. Hence, to detect good properties on $f$, we need to fix our attention to the resolution $h$ especially to a single blow-up.
\begin{remark}
\textup{Sometimes, the choice of $\widetilde{X}$ varies. For example, we can also choose $\widetilde{X}$ to be the resolution of an irreducible component of the product $X\times_Y\widetilde{Y}$ dominating $X$ and also get a commutative diagram.}
\end{remark}

\subsection{Differential forms and currents}\label{subsection2.2} In this subsection, we refer to \cite{Dem85} for the definitions and properties below.

Let $X$ be a normal compact complex space of dimension $n$. Then, $X$ is locally a closed analytic subset of an open subset of $\mathbb{C}^N$, i.e., for any $x\in X$, there exists an open neighbourhood $U$ of $x$ and an open subset $\Omega\subseteq \mathbb{C}^N$ such that $j:U\rightarrow\Omega$ is a closed embedding. \textit{A differential form $\omega$ of type $(p,q)$} on $X$ is a differential form on the smooth locus $X_{\text{reg}}$ such that for any $x\in X$, with the closed embedding $x\in U\subseteq \Omega\subseteq \mathbb{C}^N$, there exists a differential form $\alpha$ on $\Omega$ such that $\omega|_{X_{\text{reg}}}=j^*\alpha|_{X_{\text{reg}}}$ locally. We denote by $\mathcal{C}^{p,q}$ the space of differential $(p,q)$-forms.

A \textit{current} $T$ of bidegree $(p,q)$ is defined as an element in the dual space of differential (smooth) $(n-p,n-q)$-forms. We denote by  $\mathcal{D}^{p,q}(X)$ the space of $(p,q)$-currents with bidimension $(n-p,n-q)$ on $X$. It is well-defined, as well as the  closedness and positivity. %(see \cite[pp. 9]{Demailly1985measures} for more information). %Moreover, by Poincar\'e duality, Hodge theory and Stokes' theorem, we can also regard each closed $(p,q)$-current $T$ as a class of $H^{p,q}(X_{\text{reg}}, \mathbb{C})$ (cf.~\cite{dinh2015compact}).

Suppose $f:X\rightarrow Y$ is a morphism between two normal compact complex spaces of  dimensions  $m$ and $n$, respectively. 
We refer the readers to \cite{horing2016minimal} for  the pullback of differential forms by pulling back forms locally. 
Suppose further that $f$ is proper. Then, one can define the push-forward of currents $f_*$ by setting $\left<f_*T,\alpha\right>:=\left<T,f^*\alpha\right>$ for each $\alpha\in\mathcal{C}^{p,q}(X)$. For any $(p,p)$-current $T$ and $(q,q)$-form $\alpha$ on $X$, we denote by $T\wedge \alpha$ the $(p+q,p+q)$-current such that for each $(n-p-q,n-p-q)$-form $\beta$ on $X$, 
$\left<T\wedge\alpha, \beta\right>=\left<T,\alpha\wedge\beta\right>$.
In addition, since each differential $(p,q)$-form $\alpha$ on $X$ can also be naturally regarded as a $(p,q)$-current, we write $[\alpha]$ to represent a $(p,q)$-current, i.e., for each $(n-p,n-q)$-form $\beta$, $\left<[\alpha],\beta\right>:=\int_{X_{\text{reg}}}\alpha\wedge\beta$.

 We recall the following well-known projection formula which will be heavily used in this article.

\begin{pro}[Projection formula]\label{projection formula}
Suppose $f:X\rightarrow Y$ is a proper surjective morphism between normal compact complex spaces with $\dim X=m$ and $\dim Y=n$. Then for any $(l,l)$-current $T$ on $X$ and differential $(k,k)$-form $\beta$ on $Y$, we have $f_*(T\wedge f^*\beta)=f_*T\wedge\beta$ in the sense of currents. In particular, if $m=n$ (and hence $f$ is generically finite), then $f_*[f^*\beta]=\deg f\cdot [\beta]$ as currents.
\end{pro}
\begin{proof}
For any differential $(n-k-l,n-k-l)$-form $\gamma$ on $Y$, it follows from the push-forward of currents that
$$\left<f_*(T\wedge f^*\beta),\gamma\right>=\left<T\wedge f^*\beta,f^*\gamma\right>=\left<T,f^*(\beta\wedge\gamma)\right>=\left<f_*T,\beta\wedge\gamma\right>=\left<f_*T\wedge\beta,\gamma\right>,$$
which proves the first statement. Now we take two steps to prove  the latter claim. 

Step 1. From \textbf{Main Reduction}, we get a  commutative diagram  $h\circ\widetilde{f}=f\circ g$ such that $\widetilde{f}:\widetilde{X}\rightarrow\widetilde{Y}$ is a generically finite morphism of complex manifolds and $\deg f=\deg\widetilde{f}$. Then, it follows from the Gysin morphism that $\widetilde{f}_*\widetilde{f}^*\alpha=\deg f\cdot \alpha$ for each differential form $\alpha$ on $\widetilde{Y}$ (cf.~\cite[Remark 7.29]{voisin2007hodge}).

Step 2. For any$(n-k,n-k)$-form $\gamma$ on $Y$, we have
$$\left<f_*[f^*\beta],\gamma\right>=\left<[f^*\beta],f^*\gamma\right>=\int_{X_{\text{reg}}}f^*(\beta\wedge\gamma)=\int_{\widetilde{X}}g^*f^*(\beta\wedge \gamma).$$
The third equality is because a proper closed subvariety of a compact K\"ahler manifold has Lebesgue measure zero. Then,
$$\int_{\widetilde{X}}g^*f^*(\beta\wedge \gamma)=\int_{\widetilde{X}}\widetilde{f}^*h^*(\beta\wedge \gamma)=\deg\widetilde{f}\cdot\int_{\widetilde{Y}}h^*(\beta\wedge\gamma).$$
Note that the right hand side is equal to $\deg f\cdot \int_{Y_{\text{reg}}}(\beta\wedge\gamma)$ which is nothing but $\deg f\cdot \left<[\beta],\gamma\right>$. Then, our second statement holds.
\end{proof}

\subsection{Classes and cones}\label{section2.3}
In this subsection, we recall the notation and terminology which will be heavily used in this article. First we refer the readers to \cite[Definition 4.6.2]{boucksom2013an} or \cite[Section 3]{horing2016minimal} for the definition of Bott-Chern cohomology $\textup{H}_{\textup{BC}}^{1,1}(X)$ for singular spaces.

On the one hand,
for any holomorphic map $f:X\rightarrow Y$ between normal compact complex spaces, the pull-back of differential forms $f^*$ induces a well-defined morphism  $f^*:\textup{H}^{1,1}_{\textup{BC}}(Y)\rightarrow \textup{H}^{1,1}_{\textup{BC}}(X)$ (cf.~\cite[Remark 3.2]{horing2016minimal}): for any $c\in \textup{H}^{1,1}_{\textup{BC}}(Y)$,  we can write $c=\omega+dd^cu$ where $\omega$ is a $(1,1)$-form with local potentials and $u$ is a global smooth function. Then, the pull-back  $f^*c$ is  written as the form $f^*\omega+dd^c(u\circ f)$ and it is easy to see that $f^*\omega$ is a $(1,1)$-form with local potentials since the real parts of holomorphic functions of $Y$ pull back to real parts of holomorphic functions of $X$. Hence, $f^*c$ defines an element in $\textup{H}^{1,1}_{\textup{BC}}(X)$. 

On the other hand, if $X$ has at worst rational singularities in the Fujiki class $\mathcal{C}$, then it follows from the Leray sequence (cf.~\cite[Remark 3.7]{horing2016minimal}) that there is an injection: 
$\textup{H}^{1,1}_{\textup{BC}}(X)\hookrightarrow \textup{H}^2(X,\mathbb{R})$, 
and hence one can define the intersection product on the Bott-Chern cohomology space $\textup{H}^{1,1}_{\textup{BC}}(X)$ via the cup-product for $\textup{H}^2(X,\mathbb{R})$. In addition, this embedding implies that  $\textup{H}^{1,1}_{\textup{BC}}(X)$  is a finite dimensional vector space.

\begin{remark}\label{remark2.6}
\textup{Suppose $X$ is a compact K\"ahler manifold. Then by Hodge theory, there is a natural isomorphism from the Bott-Chern cohomology group to the Dolbeault group with real coefficients $\textup{H}^{1,1}_{\textup{BC}}(X)\cong \textup{H}^{1,1}(X,\mathbb{R}):=\textup{H}^{1,1}(X,\mathbb{C})\cap \textup{H}^2(X,\mathbb{R})$.}
\end{remark}

To make connections with the normal projective setting (\cite[Definition 3.6 and 3.8]{horing2016minimal}),  
let $\textup{N}_1(X)$  be the vector space of real closed currents of bidegree $(n-1,n-1)$  modulo the following equivalence: $T_1\equiv T_2$ if and only if: $T_1(\eta)=T_2(\eta)$ for all real closed $(1,1)$-forms $\eta$ with local potentials.

\begin{remark}\label{pushforward subvariety}
\textup{Suppose $f:X\rightarrow Y$ is a proper morphism between normal compact complex spaces and $T_C$ (or denoted by $[C]$) is a \textit{current of integration} (cf.~\cite[Section 3.B]{horing2016minimal}) of an irreducible curve $C$ on $X$. We claim that:  if $f(C)$ is a curve, then $f_*[C]=\deg f|_C\cdot [f(C)]$; if $C$ is contracted by $f$, then $f_*[C]=0$. This coincides with the push-forward of cycles in the projective setting.}
\end{remark}
\begin{proof}
We follow the \textbf{Main Reduction}. First, we prove that for any resolution $h:\widetilde{Y}\rightarrow Y$ with $E\subseteq\widetilde{Y}$ a contracted subvariety, $h_*[E]=0$. Indeed, for the local map $h':\widetilde{Y}\rightarrow Y\hookrightarrow \mathbb{C}^N$ and each differential $(k,k)$-form $\gamma$ on $Y$ with $\gamma=i^*\gamma_0$ and $k=\dim E$,
$$\left<h_*[E],\gamma\right>=\left<[E],h^*\gamma\right>=\left<[E],(h')^*\gamma_0\right>=0.$$
The last equality is due to $\dim h'(E)<\dim E$ and $h'$ is locally a holomorphic map between manifolds. 
Therefore, $h_*[E]=0$. 
Second, suppose the curve $C$ is contracted by $f$. Then, no matter $C\subseteq\text{Sing}\,X$ or not, any analytic curve of $\widetilde{X}$ dominating $C$ must be contracted  by either $\widetilde{f}$ or $h$. 
Hence, $f_*[C]=0$ when $C$ is contracted by $f$.

From now on, suppose $f(C)$ is a curve. Then for any differential $(n-1,n-1)$-form $\gamma$ on $Y$, we apply \textbf{Main Reduction}. If $C\not\subseteq\text{Sing}(X)$, set $\widetilde{C}\subseteq \widetilde{X}$ to be the proper transform of $C$; if $C\subseteq\text{Sing}(X)$,  let $\widetilde{C}\subseteq \widetilde{X}$ be any curve dominating $C$ (which exists since the resolution $g$ is projective). In both cases,
\begin{align*}
\left<f_*[C],\gamma\right>=\int_{C_{\text{reg}}} f^*\gamma
%=\frac{1}{\deg g|_{\widetilde{C}}}\cdot\int_{\widetilde{C}}g^*f^*\gamma
=\frac{1}{\deg g|_{\widetilde{C}}}\cdot\int_{\widetilde{C}}\widetilde{f}^*h^*\gamma
%=\frac{\deg\widetilde{f}|_{\widetilde{C}}\cdot\deg h|_{\widetilde{f}(\widetilde{C})}}{\deg g|_{\widetilde{C}}}\cdot\int_{h(\widetilde{f}(\widetilde{C}))}\gamma
=\deg f|_C\cdot\int_{f(C)}\gamma=\deg f|_C\cdot\left<[f(C)],\gamma\right>.
\end{align*}
Therefore, we complete the proof of our remark.
\end{proof}

Let $\overline{\textup{NA}}(X)\subseteq \textup{N}_1(X)$  be the closed cone generated by the classes of positive closed $(n-1,n-1)$-currents. 
We define the Mori cone $\overline{\textup{NE}}(X)\subseteq \overline{\textup{NA}}(X)$ as the closure of the cone generated by the currents of integration $T_C$, where $C\subseteq X$ is an irreducible curve.

From now on, we begin to discuss the positivity and cones in complex geometry. In the beginning, we refer the readers  to \cite[Chapter \Rmnum{1}]{Demailly2012complex} and \cite[Introduction and main results]{Demailly1992regularization} for the definition and basic properties of (strictly) plurisubharmonic (psh) functions.

Suppose $X$ is a  normal compact complex space of dimension $m$. Then, it follows from \cite[Chapter \Rmnum{1}, Theorem 5.8]{Demailly2012complex} that a smooth function $\varphi$ is psh on 
$X$ if and only if $dd^c\varphi\ge 0$ is a positive current.  %where $d^c=i(\overline{\partial}-\partial)$ and thus $dd^c=2i\partial\overline{\partial}$.

A normal compact complex space $X$ is said to be \textit{K\"ahler} if there exists a K\"ahler form $\omega$ on $X$, i.e., a  positive closed real $(1,1)$-form $\omega\in\mathcal{C}^{1,1}(X)$ such that the following holds. For every point $x\in X_{\text{sing}}$, there exist an open neighbourhood $x\in U\subseteq X$, a closed embedding $i_U:U\subseteq \Omega$ into an open set $\Omega\subseteq\mathbb{C}^N$ and a strictly psh $C^{\infty}$-function $f:\Omega\rightarrow\mathbb{C}$ such that $\omega|_{U\cap X_{\text{reg}}}=(dd^cf)|_{U\cap X_{\text{reg}}}$. 

Similar to the notation of relative ample divisor related to a projective morphism, we refer the readers to  \cite{varouchas1989kahler} for the information of relative K\"ahler metric related to a K\"ahler morphism. Now, we come  to  discuss the invariant cones.

\begin{defi}\label{defi2.11}
Let $(X,\omega)$ be a normal compact K\"ahler space, where $\omega$ is a fixed K\"ahler form. A class $[\alpha]\in \textup{H}^{1,1}_{\textup{BC}}(X)$ (where $\alpha$ is a smooth form) is said to be
\begin{enumerate}
\item[(1)] K\"ahler, if it contains a representative of a smooth K\"ahler form, i.e., there is a smooth function $\varphi$ such that $\alpha+dd^c\varphi\ge \varepsilon\omega$ for some $\varepsilon>0$;
\item[(2)] nef, if for every $\varepsilon>0$, there is a smooth function $f_{\varepsilon}$ such that $\alpha+dd^cf_{\varepsilon}\ge-\varepsilon\omega$;\item[(3)] pseudo-effective, if it contains a positive current (that may not be smooth), i.e., there exists an almost psh function $\varphi=\psi+h$ where $\psi$ is psh and $h$ is smooth on $X$ such that $\alpha+dd^c\varphi\ge 0$ in the sense of currents. 
\item[(4)] big, if it contains a K\"ahler current, i.e., there exists an almost psh function $\varphi$ such that $\alpha+dd^c\varphi\ge\varepsilon\omega$ in the sense of currents for some $\varepsilon\ge 0$.
\end{enumerate}
\end{defi}

Let $X$ be a normal compact K\"ahler space. We use the following to represent the cones in Definition \ref{defi2.11} to make connections with the projective setting.

\begin{itemize}
\item $\mathcal{K}(X)$: the cone of K\"ahler classes in $\textup{H}^{1,1}_{\textup{BC}}(X)$;
 \item $\text{Nef}(X)$: the cone of nef classes in $\textup{H}^{1,1}_{\textup{BC}}(X)$;  
\item $\mathcal{E}(X)$: the cone of pseudo-effective classes in $\textup{H}^{1,1}_{\textup{BC}}(X)$;
\item $\mathcal{E}^0(X)$: the cone of big classes in $\textup{H}^{1,1}_{\textup{BC}}(X)$.
\end{itemize}

Observe that $\mathcal{K}(X)$ is an open (convex) cone (cf.~\cite[Proposition 3.6]{graf2019finite}) and $\text{Nef}(X)\subseteq \mathcal{E}(X)$ are closed convex cones by the weak compactness of bounded sets in the space of currents (cf.~\cite[Section 6]{Demailly1992regularization}). Besides, it follows from \cite[Remark 3.12]{horing2016minimal} that $\text{Nef}(X)=\overline{\mathcal{K}}(X)$. %Furthermore,  a class $[\alpha]\in \mathcal{E}^0(X)$ if and only if it can be locally represented by a positive closed current $T$ of bidegree $(1,1)$ which satisfies $T\ge\varepsilon\omega$ in the sense of currents for some $\varepsilon>0$ and some smooth positive hermitian form $\omega$ on $X$. Such a current is called a \textit{K\"ahler current} (cf.~\cite[Definition 1.6]{Demailly2004numerical}).
Furthermore, $\mathcal{E}^0(X)$ is  the interior part of $\mathcal{E}(X)$.

We formulate the following important proposition which states that all of these cones mentioned above are preserved under suitable morphisms. Hence, we can apply the cone theory developed in \cite{meng2017building} to study the Bott-Chern cohomology space. 
\begin{pro}[Invariant cones]\label{lembignef}
Suppose $f:X\rightarrow Y$ is a proper surjective morphism of two normal compact K\"ahler spaces. Then, for any $[\alpha]\in \textup{H}_{\textup{BC}}^{1,1}(Y)$, the following hold.
\begin{enumerate}
\item[\textup{(1)}] If $[\alpha]$ is pseudo-effective, then so is $f^*[\alpha]$;
\item[\textup{(2)}] If $[\alpha]$ is nef, then so is $f^*[\alpha]$ and the converse holds when $f$ is further assumed to be a modification of normal compact K\"ahler threefolds;
\item[\textup{(3)}] If $f$ is further assumed to be finite and $[\alpha]$ is K\"ahler, then so is $f^*[\alpha]$;
\item[\textup{(4)}] If $f$ is a modification and $[\alpha]$ is big, then so is $f^*[\alpha]$;
\item[\textup{(5)}] If $f$ is generically finite and $[\alpha]$ is big, then so is $f^*[\alpha]$.
\end{enumerate}
\end{pro}

\begin{proof}
The pseudo-effective cone is invariant under any holomorphic map since the pull-back of a psh function is still psh, which gives (1). Let $\omega_X$ and $\omega_Y$ be two K\"ahler forms on $X$ and $Y$ respectively such that $f^*\omega_Y\le a\omega_X$ for some constant $a>0$. For any $\varepsilon>0$, since $[\alpha]$ is nef, there exists $g_\varepsilon\in\mathcal{C}^0(Y)$ such that $\alpha+dd^cg_\varepsilon\ge-\frac{\varepsilon}{a}\cdot\omega_Y$. Then
$$f^*\alpha+dd^c(g_{\varepsilon}\circ f)\ge-\frac{\varepsilon}{a} f^*\omega_Y\ge-\varepsilon\omega_X,$$
which proves one direction of (2). For the converse, see \cite[Lemma 3.13]{horing2016minimal}.

For  (3), see \cite[Proposition 3.5]{graf2019finite}.
%Now, we prove (3). By definition, there exists a family $(U_i,\varphi_i)_{i\in I}$  on $Y$ representing $[\alpha]$, where $\varphi_i$ is  smooth and strictly psh  on $U_i$ for all $i\in I$ and locally $\varphi_i$ is the pull-back of a smooth strictly psh function on an open subset $\Omega_i\subseteq\mathbb{C}^{N_i}$ under the local embedding $U_i\hookrightarrow \Omega_i$. Then, one may find an open cover $(V_j)_{j\in J}$ of $X$ and a map $\lambda: J\rightarrow I$ such that for each $j\in J$, $V_j\subseteq f^{-1}(U_{\lambda(j)})$. Since $f$ is a finite morphism of compact spaces, by \cite[pp.253]{vajaitu1996kahler}, there exists a global smooth function $\psi:X\rightarrow \mathbb{R}_{\ge 0}$ such that $(\varphi\circ f+\psi)|_{V_j}$ is strictly psh ($1$-convex) on each $V_j$. Moreover, since $\psi$ is global, locally $dd^c(\varphi_j\circ f+\psi)$ and $dd^c(\varphi_j\circ f)$ lie in the same cohomology class, i.e., $f^*[\alpha]$. As a result, we find a family $(V_j,(\varphi_j\circ f+\psi)|_{V_j})_{j\in J}$ which represents $f^*[\alpha]$, such that $\{V_j\}_{j\in J}$ is a cover of $X$ and $(\varphi_j\circ f+\psi)|_{V_j}$ is strictly psh on each $V_j$. Therefore, the third assertion holds.  
Now, we begin to prove (4). For each big class $[\alpha]$ on $Y$, there exists a K\"ahler form $\omega_Y$ on $Y$ and a representative $T'$ of $[\alpha]$ such that $T'\ge\omega_Y$. Therefore, it is enough to prove that the pull-back of each K\"ahler class on $Y$ is big on $X$ under any modification. 
Let $f:X\rightarrow Y$ be a modification and $\omega_Y$  a K\"ahler form on $Y$. By Hironaka's Chow lemma (cf.~\cite[Theorem 7.8]{ancona2006differential}), there exists a normal compact complex space $Z$ with $p:Z\rightarrow Y$ and $q:Z\rightarrow X$ such that $p=f\circ q$, where $p$ and $q$ are compositions of finite sequences of blowing-up along smooth centers.

We claim that $p^*[\omega_Y]$ is a big class. Indeed, we only need to consider the case when $p$ is a single blow-up. Note that a blow-up of $Y$ is locally the restriction of a blow-up of $\mathbb{C}^N$. Denote by $E$ the exceptional divisor of $p$. Then $E$ is locally the restriction of the exceptional divisor $E_i$ of the blow-up $p_i:\widetilde{U_i}\rightarrow U_i$, where $U_i$ is an open subset of $\mathbb{C}^N$. For the smooth manifold $U_i$, it follows from the calculation  in \cite[Proposition 3.24 and Lemma 3.26]{voisin2007hodge} that for sufficiently large $C_i\gg 1$, $(C_ip_i^*\omega_{U_i}-E_i)|_Y$ is a positive $(1,1)$-form (also cf.~\cite[pp.116]{ancona2006differential}). Since $Y$ is compact, there exists a sufficiently large $C\gg 1$ such that $Cp^*\omega_Y-E$ is a positive $(1,1)$-current. Hence, $p^*[\omega_Y]$ is a big class. 

Moreover, $f^*\omega_Y=q_*q^*f^*\omega_Y=q_*p^*\omega_Y$ as currents by projection formula. Since $p^*[\omega_Y]$ is big, there exists a current $T\in p^*[\omega_Y]=[p^*\omega_Y]$ such that $T\ge \omega_Z$ as currents for some K\"ahler form $\omega_Z$ on $Z$. Then, we have $q_*T\in [f^*\omega_Y]=f^*[\omega_Y]$ such that $q_*T\ge q_*\omega_Z$ as currents. In addition, for any K\"ahler form $\omega_X$ on $X$, there exists $a>0$ such that $q^*\omega_X\le a\omega_Z$. Hence, $q_*T\ge q_*\omega_Z\ge\frac{1}{a}\omega_X$ as currents and then $f^*[\omega_Y]$ is big on $Y$.

Finally, for (5), if $f$ is a generically finite surjective morphism, then taking the Stein factorization of $f$ (cf.~\cite[Theorem 1.9]{ueno2006classification}), we get a modification $f':X\rightarrow Y'$ followed by a finite morphism $f'':Y'\rightarrow Y$. For any big class $[\alpha]$ on $Y$, there exists a K\"ahler current $T\in[\alpha]$ such that $T\ge \omega_Y$ as currents for some K\"ahler form $\omega_Y$ on $Y$. Then, the pull-back $(f'')^*T\ge (f'')^*\omega_Y$ and the class of the right hand side is K\"ahler by (3). Hence, $(f'')^*[\alpha]$ is big. Further, it follows from the assertion (4) that $f^*[\alpha]=(f')^*(f'')^*[\alpha]$ is also big, and thus (5) has been checked.
\end{proof}

Recall that if $f:X\rightarrow Y$ is a surjective morphism between                                                                                                                                                                                                                                                                                                                                                                                                                                                                                                                                                                                                                                                                                                                                                                                                                                                                                                                                                                                                                                                                                                                                                                                                                                                                                                                                                                                                                                                                                                                                                                                                                                                                                                                                                                                                                                                                                                                                                                                                                                                                                                                                                                                                                                                                                                                                                                                                                                                                                                                                                                                                                                                                                                                                                                                                                                                                                                                                                                                                                                                                                                                                                                                                                                                                                                                                                                                                                                                                                                                                                                                          normal projective varieties, then it follows from the projection formula and Lefschetz hyperplane theorem that the pull-back operation gives an injection $f^*: \textup{NS}_\mathbb{R}(Y)\rightarrow \textup{NS}_\mathbb{R}(X)$ on the  N\'eron-Severi groups of $X$ and $Y$. However, in the analytic case when $X$ and $Y$ are possibly non-projective, there is no ample divisor and then Lefschetz hyperplane theorem may not be true. Nevertheless, we give the lemma below by requiring mild singularities.

\begin{lem}\label{lem2.9injection}
Suppose $f:X\rightarrow Y$ is a surjective holomorphic map between normal compact K\"ahler spaces with at worst rational singularities. 
Then $f^*:\textup{H}^{1,1}_{\textup{BC}}(Y)\rightarrow \textup{H}^{1,1}_{\textup{BC}}(X)$ is an injection.
\end{lem}

\begin{proof}
We shall apply \textbf{Main Reduction}  to prove our lemma. 
Take a resolution $h:\widetilde{Y}\rightarrow Y$ and let $\widetilde{X}$ be the resolution of an irreducible component of the product $X\times_Y\widetilde{Y}$ dominating $X$ with $\widetilde{f}:\widetilde{X}\rightarrow \widetilde{Y}$ and $g:\widetilde{X}\rightarrow X$. 
Since $\widetilde{f}$ is a surjective morphism between compact K\"ahler manifolds, the induced linear operation $\widetilde{f}^*:\textup{H}_{\textup{BC}}^{1,1}(\widetilde{Y})\rightarrow \textup{H}_{\textup{BC}}^{1,1}(\widetilde{X})$ is an injection (cf.~ \cite[Lemma 7.28]{voisin2007hodge} and Remark \ref{remark2.6}).

Since both $g$ and $h$ are resolutions, $\textup{H}_{\textup{BC}}^{1,1}(Y)$ (resp.~$\textup{H}_{\textup{BC}}^{1,1}(X)$) embeds into $\textup{H}_{\textup{BC}}^{1,1}(\widetilde{Y})$ (resp.~$\textup{H}_{\textup{BC}}^{1,1}(\widetilde{X})$); see ~\cite[Lemma 3.3]{horing2016minimal}. By the commutative diagram, the induced linear operation $f^*|_{\textup{H}_{\textup{BC}}^{1,1}(Y)}$ is an injection.
\end{proof}

In particular, we get the following proposition by setting  $X=Y$ in Lemma \ref{lem2.9injection}.

\begin{pro}\label{lem6.1isomorphism}
Let $X$ be a normal compact K\"ahler space with at worst rational singularities and $f:X\rightarrow X$ a surjective endomorphism. Then, $f$ induces an isomorphism of the Bott-Chern cohomology space  $\textup{H}_{\textup{BC}}^{1,1}(X)$.
\end{pro}

At the end of this subsection, we state the fact of spanning cones below. 
Indeed, it follows immediately from the openness of K\"ahler cone.
\begin{lem}
Suppose $X$ is a normal compact K\"ahler space. Then, the pseudo-effective cone defined in Definition \ref{defi2.11} spans the Bott-Chern cohomology space.
\end{lem}

\subsection{Polarized and int-amplified endomorphisms}
In this short subsection, we come to define the key notion in this article. 

First, we prove the next lemma which asserts that for each holomorphic self-map of a compact K\"ahler space with mild singularities, the surjectivity will imply the finiteness. 

\begin{lem}\label{finiteness of endomorphism}
Let $f:X\rightarrow X$ be a surjective holomorphic map of a normal compact K\"ahler space $X$. 
Suppose that $X$ has at worst rational singularities.
Then $f$ is finite.
\end{lem}

\begin{proof}
Since $X$ is compact, any holomorphic self-map of $X$ is proper. Besides, $f$ is a surjective self-map and thus generically finite. By taking the Stein factorization $X\rightarrow Y\rightarrow X$ and applying Hironaka's Chow lemma (cf.~\cite[Theorem 7.8]{ancona2006differential}) for the modification $X\rightarrow Y$,
we see that $f$ is dominated by a composition of finite sequences of blow-ups followed by a finite morphism $Y\rightarrow X$, which is projective. Therefore, every positive dimensional fibre of $f$ will be covered by curves. 

Suppose $C\subseteq X$ is a curve lying in the fibre of $f$. 
We fix a K\"ahler class $[\omega]$ on $X$.
On the one hand, since $f$ is proper and surjective, it follows from  Proposition \ref{lembignef} (2) that $f^*\textup{Nef}(X)\subseteq\textup{Nef}(X)$.
On the other hand, since $X$ has at worst rational singularities, by Proposition \ref{lem6.1isomorphism}, $f^*|_{\textup{H}_{\textup{BC}}^{1,1}(X)}$ is invertible; hence it maps interior part of $\textup{Nef}(X)$ to the interior part of $\textup{Nef}(X)$.
In particular, $f^*[\omega]$ is still a K\"ahler class.

Considering the composite morphism $j:\widetilde{C}\rightarrow C\hookrightarrow X$, where $\widetilde{C}$ is the normalization of the curve $C$,
we see that $j^*f^*[\omega]$ is also a K\"ahler class by Proposition \ref{lembignef} (3).
Therefore, $0<\int_{\widetilde{C}}j^*f^*\omega=\int_C f^*\omega=\left<[C],f^*\omega\right>$ by the numerical characterization of K\"ahler classes (cf.~\cite{Demailly2004numerical}). 
However, the right hand side coincides with $\left<f_*[C],\omega\right>$ by the projection formula and hence vanishes (cf.~Proposition \ref{projection formula}
 and Remark \ref{pushforward subvariety}), a contradiction.
\end{proof}

\begin{defi}\label{defipolarized}
Let $f:X\rightarrow X$ be a  surjective endomorphism of a normal compact K\"ahler space. We say that $f$ is
\begin{enumerate}
\item polarized, if there exists a K\"ahler class $[\alpha]\in \textup{H}_{\textup{BC}}^{1,1}(X)$ such that $f^*[\alpha]=q[\alpha]$ for some integer $q>1$;
\item amplified, if there exist a class $[\xi]\in \textup{H}_{\textup{BC}}^{1,1}(X)$ and a K\"ahler class $[\eta]\in \textup{H}_{\textup{BC}}^{1,1}(X)$ such that $f^*[\xi]-[\xi]=[\eta]$;
\item int-amplified, if there exist K\"ahler classes $[\xi]$ and $[\eta]$ such that $f^*[\xi]-[\xi]=[\eta]$.
\end{enumerate}
\end{defi}

\begin{remark}\label{reminitial}
\textup{We shall prove in Section \ref{section3} that, an int-amplified endomorphism in the sense of Definition \ref{defipolarized} coincides with the definition in \cite{meng2017building}, when $X$ is projective.}
\end{remark}

\section{Properties of int-amplified endomorphisms: Proof of Theorem \ref{thm1.1}\,$\sim$\,\ref{thm1.3}}\label{section3}
In this section, we study the basic properties of int-amplified endomorphisms.  Most of the following are generalizations of \cite[Section 3]{meng2017building} for the case when $X$ is projective and the proofs are similar except we apply new results developed in Section \ref{section2}.

To begin with this part, we prove Theorem \ref{thm1.1} by applying the results in \cite{meng2017building}.
As reminded by referees, we require $X$ to have at worst rational singularities to make sure the induced linear operation $f^*|_{\textup{H}_{\textup{BC}}^{1,1}(X)}$ is invertible (cf.~Proposition \ref{lem6.1isomorphism}).

\begin{proof}[Proof of Theorem \ref{thm1.1}]
It is clear that $(4)\Rightarrow (1)\Rightarrow (3)$ by letting $C=\mathcal{K}(X)$. Suppose condition $(3)$. Then, applying \cite[Lemma 3.1]{meng2017building} with $V=\textup{H}_{\textup{BC}}^{1,1}(X)$ and $C=\mathcal{E}(X)$, we get $(2)$. Finally, $(2)$ implies $(4)$ is exactly \cite[Proposition 3.2]{meng2017building}.

Next, we show the equivalence of two kinds of definitions when $X$ is projective. Indeed, it is obvious that if $f$ is  int-amplified in the sense of ample divisors, then $f$ is int-amplified in the sense of K\"ahler classes. For the converse direction, suppose $X$ is projective and $f$ is int-amplified in the sense of K\"ahler classes. Then by $(2)$,  each eigenvalue of the invertible linear operation $f^*|_{\textup{H}_{\textup{BC}}^{1,1}(X)}$ has modulus greater than $1$. Since the N\'eron Severi space $\text{NS}_{\mathbb{R}}(X)$ is a subspace of $\textup{H}_{\textup{BC}}^{1,1}(X)$, each eigenvalue of $f^*|_{\text{NS}_{\mathbb{R}}(X)}$ also has modulus greater than $1$. Hence, the converse direction holds by \cite[Theorem 1.1]{meng2017building}.
%So, we complete the proof of Theorem \ref{thm1.1}.
\end{proof}

The next lemma shows that an endomorphism being int-amplified is equivalent to its power being int-amplified. The proof is easy by using Theorem \ref{thm1.1} and Proposition \ref{lembignef}.\setlength\parskip{0pt}
\begin{lem}\label{iterationlemma}
Let $f:X\rightarrow X$ be a surjective endomorphism of a normal compact K\"ahler space with at worst rational singularities. Then the following are equivalent.
\begin{enumerate}
\item[\textup{(1)}] $f$ is int-amplified.
\item[\textup{(2)}] $f^s$ is int-amplified for any positive integer $s$.
\item[\textup{(3)}] $f^s$ is int-amplified for some positive integer $s$.
\end{enumerate}
\end{lem}
%\begin{proof}
%$(1)\Rightarrow (2)$ by repeated iteration and Proposition \ref{lembignef}. $(2)\Rightarrow(3)$ is trivial. Now, suppose $(3)$ holds.  According to Theorem \ref{thm1.1}, all the eigenvalues of $(f^s)^*|_{\textup{H}_{\textup{BC}}^{1,1}(X)}$ have modulus greater than $1$. Suppose there exists an eigenvalue $\lambda$ of $f^*|_{\textup{H}_{\textup{BC}}^{1,1}(X)}$ such that $|\lambda|\le 1$. Then, $\lambda^s$ is an eigenvalue of $(f^s)^*|_{\textup{H}_{\textup{BC}}^{1,1}(X)}$ having modulus no more than $1$, a contradiction. Thus, the second equivalent condition of Theorem \ref{thm1.1} implies $(1)$.
%\end{proof}

In the following, we shall apply Theorem \ref{thm1.1} to prove several important lemmas below. \begin{lem}\label{lem3.4}
Let $\pi:X\rightarrow Y$ be a surjective morphism of normal compact K\"ahler spaces with at worst rational singularities. Let $f:X\rightarrow X$ and $g:Y\rightarrow Y$ be two surjective endomorphisms such that $g\circ\pi=\pi\circ f$. If $f$ is int-amplified, then so is $g$.
\end{lem}

\begin{proof} 
It follows from Lemma \ref{lem2.9injection} that $\pi^*: \textup{H}_{\textup{BC}}^{1,1}(Y)\rightarrow \textup{H}_{\textup{BC}}^{1,1}(X)$ is injective. Then, each eigenvalue of $g^*|_{\textup{H}^{1,1}_{\textup{BC}}(Y)}$ is an eigenvalue of  $f^*|_{\textup{H}^{1,1}_{\textup{BC}}(X)}$ and thus has modulus greater than one, since $f$ is int-amplified. Therefore, our lemma follows from Theorem \ref{thm1.1}.
\end{proof}

\begin{lem}\label{lem3.333}
Let $\pi:X\rightarrow Y$  be a generically finite surjective morphism between normal compact K\"ahler spaces with at worst rational singularities. Let $f:X\rightarrow X$ and $g:Y\rightarrow Y$  be surjective endomorphisms such that $g\circ\pi=\pi\circ f$. If $g$ is int-amplified, then so is $f$.
\end{lem}
\begin{proof}
Suppose $g$ is int-amplified. Then, there exist K\"ahler classes $[\xi]$ and $[\eta]$ on $Y$ such that $g^*[\xi]=[\eta]+[\xi]$. Note that the K\"ahler cone $\mathcal{K}(Y)$ is contained in the big cone $\mathcal{E}^0(Y)$ and big classes pull back to big classes under any generically finite morphism by Proposition \ref{lembignef}. Then, $\pi^*[\xi]$ and $\pi^*[\eta]$ are big classes on $X$ and  the commutative diagram implies $\pi^*[\eta]=f^*\pi^*[\xi]-\pi^*[\xi]$. Hence, our lemma follows from Theorem \ref{thm1.1}.
\end{proof}

Next, we follow the idea of  \cite{meng2017building} to prove Theorems \ref{thm1.2} and  \ref{thm1.3} below in our present setting. Both of them are consequences of Theorem \ref{thm1.1}.

\begin{proof}[Proof of Theorem \ref{thm1.2}]
We fix a norm on the Bott-Chern cohomology space $V:=\textup{H}_{\textup{BC}}^{1,1}(X)$. Let $\phi_f:=f^*|_V$ and $\phi_g:=g^*|_V$. Since $f$ is int-amplified, by Theorem \ref{thm1.1}, all the eigenvalues of $\phi_f$ are of modulus greater than $1$. Recall that the norm of a linear operator is given by
$||\phi_f||:=\max\limits_{v\in V}\frac{||\phi_f(v)||}{||v||}$, 
and the right hand side is no less than the spectral radius of $\phi_f$. Also, by Gelfand's formula, $\lim\limits_{i\rightarrow\infty}||\varphi_f^{-i}||^\frac{1}{i}=\rho(\varphi_f^{-1})<1$, where $\rho(\varphi_f^{-1})$ is the spectral radius of $\varphi_f^{-1}$. 
Then,  there exists $i_0$ such that for all $i\ge i_0$,
$||\phi_f^{-i}\phi_g^{-1}||\le||\phi_f^{-i}||\cdot||\phi_g^{-1}||<1$.
 So, for each $v\in V$, $||\phi_g\phi_f^i(v)||>||v||$. Let $h=f^i\circ g$ with $i\ge i_0$ and $\phi_h=h^*|_V$. 

We claim that  all the eigenvalues of $\phi_h$ are of modulus greater than $1$. Indeed, let $r$ be the smallest absolute value of eigenvalues of $\phi_h$. Since $X$ is K\"ahler, $V$ has nonvoid interior. So, it follows from the Perron-Frobenius theorem that $V$ contains an eigenvector $v$ whose eigenvalue is the spectral radius of $\phi_h^{-1}$, i.e., $||\phi_h^{-1}(v)||=\frac{1}{r}||v||$. Therefore, with $v$ replaced by $\phi_h(v)$, we have $r>1$ by the inequality:
$r||v||=||\phi_h(v)||=||\phi_g(\phi_f^i(v))||>||v||$.
Hence $h$ is int-amplified. Similar argument works for $g\circ f^i$. So, Theorem \ref{thm1.2} holds.	
\end{proof}

\begin{proof}[Proof of Theorem \ref{thm1.3}.]
By the ramification divisor formula, $f^*(-K_X)-(-K_X)=R_{\phi}$, where  the ramification divisor $R_{\phi}$ is effective and thus pseudo-effective in the sense of currents. By Theorem \ref{thm1.1} (4), $-K_X$ is also pseudo-effective.

Suppose the contrary that $\kappa(X)>0$. Then, there exists a non-zero effective  divisor $D$ on $X$, such that  $sK_X\sim D$ for some  $s\in\mathbb{N}^+$. Note that $D$ is not contained in the singular locus due to the normality of $X$. Therefore, for any K\"ahler form $\omega$ on $X$, 
$$sK_X\wedge\omega^{n-1}=[D]\wedge\omega^{n-1}=\int_{D_{\text{reg}}}(\omega|_D)^{n-1}>0,$$ 
a contradiction to $-K_X$ being pseudo-effective. Then, Theorem \ref{thm1.3} holds.
\end{proof}

Let $f:X\rightarrow X$ be a  surjective endomorphism of a normal compact K\"ahler space $X$ of dimension $n$ with at worst rational singularities (cf.~\cite[Remark 3.7]{horing2016minimal}). Denote by \setlength\parskip{0pt}
$$\textup{L}^k(X):=\{\sum [\alpha_1]\cup\cdots\cup[\alpha_k]~|~[\alpha_i]\in \textup{H}_{\text{BC}}^{1,1}(X)\},$$
the subspace of $\textup{H}^{2k}(X,\mathbb{R})$. In particular, $\textup{L}^1(X)=\textup{H}_{\text{BC}}^{1,1}(X)$. Let $\textup{N}^k(X):=\textup{L}^k(X)/\equiv_w$, where  $[\alpha]\in \textup{L}^k(X)$ is weakly numerically equivalent (denoted by $\equiv_w$) to zero if and only if for any $[\beta_{k+1}],\cdots,[\beta_n]\in \textup{H}_{\text{BC}}^{1,1}(X)$, $[\alpha]\cup[\beta_{k+1}]\cup\cdots\cup [\beta_n]=0$. Moreover, for any $[\alpha],[\beta]\in \textup{L}^k(X)$, $[\alpha]\equiv_w [\beta]$ if and only if $[\alpha]-[\beta]\equiv_w 0$.

Let  $\textup{L}_{\mathbb{C}}^k(X)=\textup{L}^k(X)\otimes_{\mathbb{R}}\mathbb{C}$ and $\textup{N}_{\mathbb{C}}^k(X)=\textup{N}^k(X)\otimes_{\mathbb{R}}\mathbb{C}$.  Note that the pull-back operation $f^*$ on Bott-Chern cohomology space induces a linear operation on  the subspace $\textup{L}^k_{\mathbb{C}}(X)\subseteq H^{2k}(X,\mathbb{C})$ for each $k$. With the notation settled, it follows from Proposition \ref{lem6.1isomorphism} that $f^*$ also gives a well-defined linear operation on the quotient space $\textup{N}_{\mathbb{C}}^k(X)$. 

We generalize \cite[Lemma 3.6]{meng2017building} to the following. We skip its proof and readers may refer to \cite[Lemma 3.6]{meng2017building} by applying Theorem \ref{thm1.1} and Proposition \ref{lem6.1isomorphism}. For its application, see Theorem \ref{thmdegree1} and Proposition \ref{torus2}.

\begin{lem}\label{degreelem}
Let $f:X\rightarrow X$ be an int-amplified endomorphism of a normal compact K\"ahler space of dimension $n$. Suppose that $X$ has at worst rational singularities. Then, for each $0<k<n$,  all the eigenvalues of $f^*|_{\textup{N}^k_{\mathbb{C}}(X)}$ are of modulus less than $\deg f$. In particular, $\lim\limits_{i\rightarrow+\infty}\frac{(f^i)^*[x]}{(\deg f)^i}\equiv_w 0$ for any $[x]\in \textup{L}^k_{\mathbb{C}}(X)$. 
\end{lem}

\begin{thm}\label{thmdegree1}
Let $f:X\rightarrow X$ be an int-amplified endomorphism of a positive dimensional compact K\"ahler space with at worst rational singularities. Then $\deg f>1$.
\end{thm}

\begin{proof}
If $\dim X=1$, then $X$ is normal since $X$ has at worst rational singularities.  
Therefore, $X$ is smooth and projective. 
By Theorem \ref{thm1.1}, there exist two ample divisors $D$ and $H$ such that $f^*D-D=H$. If $\deg f=1$, taking the degree of both sides, we get a contradiction. 

If $\dim X\ge 2$, then by Theorem \ref{thm1.1}, all the eigenvalues of $f^*|_{\textup{H}^{1,1}_{\text{BC}}(X)}$ have modulus greater than $1$. Note that the eigenvectors of the quotient space $\textup{N}^1_\mathbb{C}(X)$ come from the elements in $\textup{H}^{1,1}_{\text{BC}}(X)$. By applying Lemma \ref{degreelem} for $k=1$, $\deg f>1$.
\end{proof}

\section{Some dynamics of int-amplified endomorphisms}\label{section4}
In this section, we discuss some dynamics for int-amplified endomorphisms in terms of fixed points and totally invariant  subvarieties. Note that Lemma \ref{productint} follows immediately from \cite[Proposition 1.3.1]{varouchas1989kahler} and Lemma \ref{countable periodic point} is an extension of \cite[Lemma 2.4]{meng2017building}.
\begin{lem}\label{productint}
If $f$ is an int-amplified endomorphism of a normal compact K\"ahler space $X$ and $Z$ is an $f$-invariant analytic subvariety (i.e., $f(Z)=Z$) of $X$, then the restriction $f|_Z$ is an int-amplified endomorphism of $Z$.
\end{lem}

\begin{lem}\label{countable periodic point}
Let $f:X\rightarrow X$ be an int-amplified endomorphism of a compact K\"ahler space $X$. Then, the set of periodic points $\textup{Per}(f)$ is countable.
\end{lem}
\begin{proof}
Suppose $\text{Per}(f)$ is uncountable. Then, there exists some $s>0$, such that the set $P_s$ of all $f^s$-fixed points is infinite. Let $Z$ be an irreducible component of the closure of $P_s$ in $X$ with $\dim Z>0$. Then $f^s|_Z=\textup{id}_Z$, which is absurd by Lemmas \ref{iterationlemma} and \ref{productint}.
\end{proof}

Unlike the normal projective setting, for a compact K\"ahler space $X$ admitting an amplified endomorphism $f$, it is still  open about the density of periodic points. Nevertheless, the proposition below partially answers Question \ref{question1.6} in the case of complex tori. Further, we will show in Section \ref{section5} that the result can be extended to $Q$-tori case.

\begin{pro}\label{fixedpointtorus}
Let $f:T\rightarrow T$ be an int-amplified endomorphism of a complex torus $T$. Then, there exists a fixed point of $f$.
\end{pro}

\begin{proof}
By K\"unneth formula, %$H^i(T,\mathbb{C})$ is the $i$-th exterior power of differential $1$-forms in $H^1(T,\mathbb{C})$. So 
if $\lambda_1,\cdots,\lambda_k$ (may not be distinct) are all the eigenvalues of $f^*|_{\textup{H}^1(T,\mathbb{C})}$, then each eigenvalue of $f^*|_{\textup{H}^i(T,\mathbb{C})}$ is of the form $\prod_j\lambda_j^{l_j}$, where $\sum l_j=i$.

Since $f$ is int-amplified, each eigenvalue of $f^*|_{\textup{H}^{1,1}(T,\mathbb{C})}$ has modulus greater than $1$. By Hodge theory, all of the eigenvalues of $f^*|_{\textup{H}^1(T,\mathbb{C})}$ (and hence $f^*|_{\textup{H}^i(T,\mathbb{C})}$) are of modulus greater than $1$. Then, it follows from the Lefschetz fixed point theorem that,
$$L(f)=\sum_{i}(-1)^i\text{Trace}(f^*|_{\textup{H}^i(T,\mathbb{C})})=\sum_{k_1,\cdots,k_i}(-\lambda_{k_1})\cdots(-\lambda_{k_i})=\prod_i (1-\lambda_i),$$
with $|\lambda_i|>1$ for each $i$. Then $L(f)\neq 0$ implies that $f$ has a fixed point.
\end{proof}

In the following, we fix our attention to manifolds temporarily. We extend some previous results in \cite[Section 3]{meng2017building} to the analytic case. Recall the definition of $\textup{L}^k(X)$ and its quotient $\textup{N}^k(X)$  in Section \ref{section3}. The following lemma  holds due to the openness of the cone $\mathcal{K}_k(X)$, the set of classes of strictly positive closed $(k,k)$-forms (cf.~\cite[Chapter \Rmnum{3}]{Demailly2012complex}).
\begin{lem}\label{lemopenclosed}
Let $(X,\omega)$ be a compact K\"ahler manifold of dimension $n$. Then, for any closed  $(k,k)$-current $T$, there exists $A>0$ such that $A\omega^k-T$ is a positive $(k,k)$-current. 
\end{lem}
When $X$ is a normal projective variety, one may use Bertini's theorem to prove that for any closed subvariety $Z$, there exists an effective cycle $C$ such that the sum $Z+C$ is the intersection of ample divisors. However, in the analytic setting, we may not have ample divisors. Nevertheless, Lemma \ref{lemopenclosed} says that for any closed current, there exists a positive ``complement'' such that the sum is (numerically) equal to a power of some K\"ahler form. As an application of Lemma \ref{lemopenclosed}, we show the lemma below.  Note that when $X$ is smooth,  $\textup{H}_{\textup{BC}}^{1,1}(X)\cong \textup{H}^{1,1}(X,\mathbb{R})$ and hence we may regard each $(n-k,n-k)$-current as a function on $\textup{L}^k(X)$ using the Poincar\'e duality and Stokes' formula. 

%\begin{lem}\label{lemlimitofintegration}
%Let $(X,\omega)$ be a compact K\"ahler manifold of dimension $n$. Suppose that there exists a sequence $[x_i]\in \textup{L}^k(X)$ with $0<k<n$ such that $\left<T,[x_i]\right>\ge 0$ for any positive closed $(n-k,n-k)$-current $T$. Suppose further that $\lim\limits_{i\rightarrow +\infty}[x_i]\equiv_w0$ in $\textup{L}^k_{\mathbb{C}}(X)$. Then, $\lim\limits_{i\rightarrow +\infty}\left<T,[x_i]\right>=0$ for any  $(n-k,n-k)$-current $T$.
%\end{lem}

%\begin{proof}
%We apply Lemma \ref{lemopenclosed} for both $T$ and $-T$. Then, there exist  positive closed $(n-k,n-k)$-currents $T_1$ and $T_2$ such that $T+T_1=A\omega^{n-k}$ and $-T+T_2=A\omega^{n-k}$ in the sense of currents, where $\omega$ is the given K\"ahler form and $A$ is a positive number. Therefore,
%\begin{align*}
%\varlimsup_{i \rightarrow \infty}\left<T,[x_i]\right>\le \varlimsup_{i\rightarrow\infty}\left<T+T_1, [x_i]\right>=\lim_{i\rightarrow\infty}  A\omega^{n-k}\cup [x_i]=0
%\end{align*}
%by the definition of weak numerical equivalence. Similarly, 
%$\varlimsup\limits_{i \rightarrow \infty}\left<-T,x_i\right>\le 0$, and thus
%$0\le\varliminf\limits_{i \rightarrow \infty}\left<T,x_i\right>\le\varlimsup\limits_{i \rightarrow \infty}\left<T,x_i\right>\le 0$. 
%Hence, our result holds.
%\end{proof}

In what follows, we show a useful lemma, which will be applied in the proof of Proposition \ref{propersubvariety} and Lemma \ref{keylem2finite}. 
\begin{lem}\label{lemvanish}
Let $f:X\rightarrow X$ be an int-amplified endomorphism of a normal compact K\"ahler space with at worst  rational singularities of dimension $n$. Let  $\xi$  be a K\"ahler form on $X$. Then for any $k$-dimensional proper analytic subvariety $Z$ of $X$, which is not contained in the singular locus,  $\lim\limits_{i\rightarrow +\infty} \int_Z \frac{(f^i)^*(\xi^k)}{(\deg f)^i}=0$.
\end{lem}
\begin{proof}
Let $x_i:=\frac{(f^i)^*(\xi^k)}{(\deg f)^i}\in \textup{L}^k(X)$. We take a resolution $\pi:\widetilde{X}\rightarrow X$ and consider the pull-back $y_i:=\pi^*x_i$, the class of which is a product of nef classes in $\widetilde{X}$. Assume that outside a proper closed subset $E\subseteq\widetilde{X}$, we have $\widetilde{X}\backslash E\cong X\backslash\text{Sing}(X)$. Since $Z\not\subseteq\text{Sing}(X)$, we denote by $\widetilde{Z}$  the proper transform of $Z$ in $\widetilde{X}$.

For any fixed K\"ahler form $\omega$ on $X$, there exists  $c\gg 1$ such that $\eta:=c\cdot\pi^*\omega-E'$ is a K\"ahler form on $\widetilde{X}$ (cf. \cite[pp.81]{voisin2007hodge}), where $E'\ge 0$ is a positive linear combination of the components in $E$.
We may identify $E'$ with $E$ in the following.
Since $[\widetilde{Z}]$ is a positive $(n-k,n-k)$-current and $[y_i]$ is a product of nef classes, by Lemma \ref{lemopenclosed}, there exists  $a>0$ such that  $\left<a\cdot\eta^{n-k}-[\widetilde{Z}],[y_i]\right>\ge 0$ for any $i$; hence
$$\int_Z x_i=\int_{\widetilde{Z}}y_i=\left<[\widetilde{Z}],[y_i]\right>\le [y_i]\cup a\cdot\eta^{n-k}\le [y_i]\cup (ac)\pi^*\omega\cup\eta^{n-k-1}.$$
Note that the last inequality is due to the nefness of $[y_i]$ and $[\eta]$ (and hence $[y_i]\cup E\cup \eta^{n-k-1}\ge 0$).
Repeat the above process and apply the nefness of $[\pi^*\omega]$, we have
\begin{equation}\label{positive part}
	\int_Z x_i\le a\cdot c^{n-k}[y_i]\cup (\pi^*\omega)^{n-k}=a
	\cdot c^{n-k}[x_i]\cup\omega^{n-k}.
\end{equation}

With $\widetilde{Z}$ replaced by $-\widetilde{Z}$, the same arguments above together with Lemma \ref{lemopenclosed} show that there exists $b>0$ such that $\left<b\cdot\eta^{n-k}+[\widetilde{Z}],[y_i]\right>\ge 0$ for any $i$; hence
\begin{equation}\label{negative part}
	-\int_Zx_i=-\int_{\widetilde{Z}}y_i\le[y_i]\cup b\cdot \eta^{n-k}\le  b\cdot c^{n-k}[x_i]\cup\omega^{n-k}.
\end{equation}
By Lemma \ref{degreelem},  $\lim\limits_{i\rightarrow\infty} [x_i]\equiv_w0$ in $\textup{L}_{\mathbb{C}}^k(X)$;
hence taking the upper limits of both sides for  (\ref{positive part}) and (\ref{negative part}), we have
$0\le\varliminf\limits_{i \rightarrow \infty}\left<[Z],x_i\right>\le\varlimsup\limits_{i \rightarrow \infty}\left<[Z],x_i\right>\le 0$. So our result follows. 
\end{proof}

\begin{pro}\label{propersubvariety}
Let $f:X\rightarrow X$ be an int-amplified endomorphism of a normal compact K\"ahler space $X$ with at worst rational singularities of dimension $n$. Let $Y$ be an analytic closed subvariety of $X$ not contained in the singular locus such that $f^{-1}(Y)=Y$ and $\deg f|_Y=\deg f$. Then $X=Y$.
\end{pro}

\begin{proof}
Suppose $\dim Y=m<n$ and we fix a K\"ahler form $\xi$ on $X$. On the one hand, by Lemma \ref{lemvanish}, $\lim\limits_{i\rightarrow \infty}\int_Y\frac{(f^i)^*\xi^m}{(\deg f)^i}=0$. On the other hand, we have the natural integration on the topological space $\left<[Y],(f^i)^*\xi^m\right>=(\deg f|_Y)^i\left<[Y],\xi^m\right>=(\deg f)^i\left<[Y],\xi^m\right>$.

Therefore, we get a contradiction by the numerical characterization (cf.~\cite{Demailly2004numerical}):
$$0<\int_{Y_\text{reg}}\xi^m=\int_Y\xi^m=\frac{1}{(\deg f)^i}\int_Y(f^i)^*\xi^m=\lim_{i\rightarrow\infty}\int_Y\frac{(f^i)^*\xi^m}{(\deg f)^i}=0.$$
Then, our proposition follows from this contradiction.
\end{proof}

Resulting from Proposition \ref{propersubvariety}, we end up this section with the following lemma.

\begin{lem}\label{totoallyperoidic}
Let $f:T\rightarrow T$ be an int-amplified endomorphism of a complex torus $T$. Let $Z$ be a non-empty $f^{-1}$-periodic closed analytic subvariety of $T$. Then $Z=T$.
\end{lem}
\begin{proof}
Since $f^s$ is int-amplified for any positive integer $s$ (cf.~Lemma \ref{iterationlemma}), with $f$ replaced by its power, we may assume $Z$ is irreducible and $f^{-1}(Z)=Z$. It follows from the ramification divisor formula and the purity of branch loci that $f$ is \'etale. Therefore we get $\deg f|_Z=\deg f$; and hence $Z=T$ by Proposition \ref{propersubvariety}.
\end{proof}

\section{The quotients of complex tori}\label{section5}
In this section, we deal with the case of $Q$-tori admitting int-amplified endomorphisms. Recall that a normal compact K\"ahler space $X$ is said to be a $Q$-\textit{torus} if there exists a complex torus (full rank) $T$ and a finite surjective morphism $\pi: T\rightarrow X$ such that $\pi$ is \'etale in codimension $1$ (cf.~\cite[Notation 2.1]{horing2011non}). Besides, we say that a finite surjective morphism $f:X\rightarrow Y$ is \textit{quasi-\'etale} if $f$ is \'etale in codimension $1$. 
It follows from the ramification divisor formula that any surjective endomorphism of a $Q$-torus is quasi-\'etale.

To begin with this section, we extend the results in \cite[Lemma 2.12]{nakayama2010polarized} and \cite[Proposition 4.3]{nakayama2009building} to the following. Readers may refer to \cite[Proposition 3]{Beauville1983some}  for its proof.

\begin{pro}\label{albanese closure}
Let $X$ be a $Q$-torus. Then there exists a quasi-\'etale cover $\tau:T\rightarrow X$ satisfying the following conditions:
\begin{enumerate}
\item $T$ is a complex torus;
\item $\tau$ is Galois;
\item If $\tau':T'\rightarrow X$ is another quasi-\'etale cover from a torus $T'$, then there exists an \'etale morphism $\sigma:T'\rightarrow T$ such that $\tau'=\tau\circ\sigma$.
\end{enumerate}
In particular, we call the quasi-\'etale cover $\tau$ in Proposition \ref{albanese closure} the \textit{Albanese closure} of $X$ in codimension one.
\end{pro}

As a consequence of Proposition \ref{albanese closure}, we have the lifting lemma below. Then, Corollary \ref{coro5.4} follows immediately from Proposition \ref{fixedpointtorus} and Lemma \ref{etalelifting}. 
\begin{lem}\label{etalelifting}
Let $f:X\rightarrow X$ be a surjective endomorphism of a $Q$-torus $X$. Then there exist a complex torus $T$, a quasi-\'etale morphism $\tau: T\rightarrow X$ and an \'etale endomorphism $\sigma_T:T\rightarrow T$ such that $\tau\circ \sigma_T=f\circ \tau$.
\end{lem}
\begin{proof}
Note that $f$ is quasi-\'etale by the ramification divisor formula.
Let $\tau: T\rightarrow X$ be the Albanese closure as in Proposition \ref{albanese closure}. Since the composition $f\circ\tau$ is still quasi-\'etale, by Proposition \ref{albanese closure} (3), such $\sigma_T$ exists.
\end{proof}
\begin{coro}\label{coro5.4}
Let $f:X\rightarrow X$ be an int-amplified endomorphism of a Q-torus. Then, there exists a fixed point of $f$.
\end{coro}

Comparing with Lemma \ref{totoallyperoidic}, we have the same argument for $Q$-tori. For the application of Lemma \ref{lemma5.6}, see the proof of Theorem \ref{thmmmpthreefold} in Section \ref{section8}.

\begin{lem}\label{lemma5.6}
Let $f:X\rightarrow X$ be an int-amplified endomorphism of a $Q$-torus $X$. Then, there is no $f^{-1}$-periodic proper subvariety of $X$.
\end{lem}

\begin{proof}
Let $Z$ be a non-empty $f^{-1}$-periodic closed subset of $X$, i.e., $f^{-s}(Z)=Z$ for some $s>0$. One needs to show that $Z=X$. Since $X$ is a finite quotient of some complex torus $T$, it has only quotient and hence rational (cf.~\cite[Proposition 5.15]{kollar2008birational}) singularities. Also, it follows from the ramification divisor formula that $f$ is quasi-\'etale. By Lemma \ref{etalelifting}, there exists a quasi-\'etale cover $\tau: T\rightarrow X$ from a complex torus $T$. Since $f$ is quasi-\'etale, so is the composition $f\circ\tau$. Then, one gets a  surjective endomorphism $\sigma_T:T\rightarrow T$ such that $f\circ\tau=\tau\circ \sigma_T$.  By Lemma \ref{lem3.333}, $\sigma_T$ is int-amplified. Further, the commutative diagram gives that $\sigma_T^{-s}(\tau^{-1}(Z))=\tau^{-1}(Z)$. Since  $\sigma_T^s$ is also int-amplified, $\tau^{-1}(Z)=T$ (cf.~Lemma \ref{totoallyperoidic}) and hence $Z=X$.
\end{proof}

From now on, we aim to show that each dominant meromorphic map from a normal compact K\"ahler space with mild singularities to a $Q$-torus is holomorphic. We first %state the following extension lemma (cf.~\cite[Lemma 9.11]{ueno2006classification})  for the convenience of readers. 
%\begin{lem}[\cite{ueno2006classification}]\label{lemtoruskey1}
%Suppose $f:X\dashrightarrow T$ is a dominant meromorphic map from a  compact K\"ahler manifold  to a complex torus $T$. Then, $f$ is a morphism.
%\end{lem}
%\begin{proof}
%Let $\pi: \mathbb{C}^n\rightarrow T$ be a universal covering of $T$. %Let $S$ be the union of the singular locus of $X$ % (discrete since $X$ is smooth in codimension two) %
%Let $S$ be the set of points of indeterminacy of the meromorphic map $f$. Since $X$ is normal, by \cite[Theorem 2.5]{ueno2006classification}, $S\subseteq X$ is an analytic subset of codimension at least two. For a point $P\in S$, we take a simply-connected small open neighborhood $U$ of $P$ in $X$. Then, $U'=U\backslash S\cap U$ is also simply-connected. Hence there exists a morphism $f':U'\rightarrow \mathbb{C}^n$ such that $\pi\circ f'=f|_{U'}$. Besides, $f'$ can be considered as an $n$-tuple $(f_1',\cdots,f_n')$ of holomorphic functions on $U'$. As $S\cap U$ is of at least codimension two, using Hartog's theorem, $f_i'$ can be extended to a holomorphic function $\widetilde{f_i}$ on $U$. As $(\widetilde{f_1},\cdots,\widetilde{f_n})$ define a holomorphic map $\widetilde{f}:U\rightarrow \mathbb{C}^n$, $\pi\circ\widetilde{f}$ defines a holomorphic map of $U$ into $T$. Hence $f$ can be naturally extended to a holomorphic map of $X$ to $T$.
%\end{proof}
%In the following, we 
give a generalized version of \cite[Lemma 9.11]{ueno2006classification} by weakening the condition. Readers are referred to \cite[Lemma 8.1]{kawamata1985minimal} for the projective case but the proofs are different. To prove Lemma \ref{lemtoruskey2}, we need to recall the \textit{Albanese torus} for a singular complex space (\cite[Section 3]{graf2018algebraic}).

\begin{lem}\label{lemtoruskey2}
Let $f:X\dashrightarrow T$ be a dominant meromorphic map from a normal compact K\"ahler space with only rational singularities to a torus $T$. Then, $f$ is a morphism.
\end{lem}
 \begin{comment}
%\begin{defi}[\cite{graf2018algebraic}]\label{albanese map}
%

\[\xymatrix{X\ar@{-->}[rr]\ar[dr]_{\alpha}&&T\\
&A\ar@{-->}[ur]_{\exists~!}&
}\]\end{defi}
\end{comment}

Let $X$ be a compact complex space. A holomorphic map $\alpha:X\rightarrow A$ to a complex torus $A$ is said to be the \textit{Albanese map} of $X$, and $A$ is called the \textit{Albanese torus} of $X$, if any map from $X$ to a complex torus $T$ factors uniquely through $\alpha$. Note that the Albanese torus is unique up to isomorphism if it exists. Let $A=\text{Alb}(X)$ and $\alpha=\text{alb}_X$. If $X$ is further assumed to have at worst rational singularities, then $X$ admits an Albanese torus and $\text{Alb}(X)=\text{Alb}(\widetilde{X})$, where $\pi: \widetilde{X}\rightarrow X$ is a projective desingularization such that $\text{alb}_X\circ \pi=\text{alb}_{\widetilde{X}}$ (cf.~\cite[Theorem 3.4]{graf2018algebraic}). 
%The following lemma follows from \cite[Lemma 3.5]{graf2018algebraic} which is a key to the proof of Lemma \ref{lemtoruskey2}. 
%We refer readers to \cite[Theorem 3.45]{Kollar2007leture} for the existence of a projective resolution for any compact complex space. 
Now, we are ready to prove Lemma \ref{lemtoruskey2}.

%\begin{lem}\label{keykeyrational}
%Let $X$ be a normal compact complex space with at worst rational singularities and $\pi:\widetilde{X}\rightarrow X$ is a projective desingularization. Then, any map $\varphi:\widetilde{X}\rightarrow T$ to a complex torus $T$ factors through $\pi$.
%\end{lem}
\begin{proof}[Proof of Lemma \ref{lemtoruskey2}]
Let $\pi:\widetilde{X}\rightarrow X$ be a projective resolution, $\widetilde{f}:\widetilde{X}\dashrightarrow T$ the induced dominant meromorphic map, $\text{alb}_X:X\rightarrow \text{Alb}(X)$ and $\beta:\text{Alb}(X)\dashrightarrow T$ the induced map (cf.~\cite[Section 3]{graf2018algebraic}). %To be more precise, the following diagram is commutative.
%\[\xymatrix{\widetilde{X}\ar[d]_{\pi}\ar@{-->}[dr]^{\widetilde{f}}&\\
%X\ar[d]_{\text{alb}_X}\ar@{-->}[r]^f&T\\
%\text{Alb}(X)\ar@{-->}[ur]_\beta &
%}\]
Since $\widetilde{f}$ is a dominant meromorphic map from a complex manifold to a torus, by \cite[Lemma 9.11]{ueno2006classification}, $\widetilde{f}$ is a morphism. Moreover, $\text{Alb}(X)=\text{Alb}(\widetilde{X})$ and thus $\beta$ is a morphism by \cite[Definition 9.6 and Theorem 9.7]{ueno2006classification}. Therefore, $f$ is a morphism.
\end{proof}

We end up this section with the following proposition. The proof is the  same as in the projective setting  (cf.~\cite[Lemma 5.3]{meng2018building} or \cite[Lemma 4.4]{meng2017building}) except that we apply Lemma \ref{lemtoruskey2} instead of \cite[Lemma 5.1]{meng2018building}.  

%Lemma \ref{equiirred} instead of \cite[Lemma 5.2]{meng2018building} and apply 

\begin{pro}\label{proholomorphic}
Let $\pi:X\dashrightarrow Y$ be a proper dominant meromorphic map from a normal compact K\"ahler space $X$ to a $Q$-torus $Y$. Suppose that $X$ has at worst Kawamata log terminal (klt) singularities and the normalization of the graph $\Gamma_{X/Y}$ is equi-dimensional over $Y$ (this holds when the general fibre of $\pi$ is connected, $f:X\rightarrow X$ is int-amplified and $f$ descends to some int-amplified endomorphism $g:Y\rightarrow Y$). Then, $\pi$ is a morphism.
\end{pro}

\section{$K_X$ pseudo-effective case: Proof of Theorem \ref{thmtorus}}\label{section6}
Proceeding from this section, we will gradually fix our attention to the normal compact K\"ahler threefolds. In Section \ref{section8}, we will discuss the equivariant MMP for threefolds and before that,  some preparations are necessary. In this section, we consider the case when the canonical divisor is pseudo-effective. 
Now, we are ready to prove Theorem \ref{thmtorus}. We first divide the theorem into three parts (cf.~Propositions \ref{torus2}, \ref{torus1} and \ref{torus5}) and then, we put them together to show Theorem \ref{thmtorus} for the convenience of readers.

\begin{pro}\label{torus2}
Let $f:X\rightarrow X$ be an int-amplified endomorphism of a compact K\"ahler manifold of dimension $n\ge 1$ with pseudo-effective $K_X$. Then $X$ is a $Q$-torus.
\end{pro}
\begin{proof}
Since $-K_X$ is pseudo-effective by Theorem \ref{thm1.3} and $K_X$ is pseudo-effective by assumption, $K_X\equiv 0$ in the sense of currents. Then, $f$ is \'etale by the ramification divisor formula and Purity of the branch locus theorem of Grauert-Remmert (cf.~\cite{grauert1955zur}).

Fix a K\"ahler class $[\omega]$ on $X$. We claim that $c_2(X)\cdot [\omega]^{n-2}=0$. Indeed, according to \cite[Proposition 5.6]{graf2019finite},
$(\deg f) c_2(X)\cdot [\omega]^{n-2}=c_2(X)\cdot (f^*[\omega])^{n-2}=c_2(X)\cdot f^*([\omega]^{n-2})$.
Then, $$c_2(X)\cdot[\omega]^{n-2}=c_2(X)\cdot\frac{f^*([\omega]^{n-2})}{\deg f}=c_2(X)\cdot\frac{(f^i)^*([\omega]^{n-2})}{(\deg f)^i}=:c_2(X)\cdot [x_i].$$
Here, $\lim\limits_{i\rightarrow\infty}[x_i]\equiv_w 0$ by Lemma \ref{degreelem}. Moreover, applying Lemma \ref{lemopenclosed} for both $c_2(X)$ and $-c_2(X)$, there exist two positive closed $(2,2)$-currents $T_1$ and $T_2$ and a positive number $A$ such that $c_2(X)+T_1=A[\omega]^2$ and $-c_2(X)+T_2=A[\omega]^2$ in the sense of currents. Then, we get the inequalities:
$\varlimsup\limits_{i \rightarrow \infty} c_2(X)\cdot [x_i]\le 0$ and
$\varlimsup\limits_{i \rightarrow \infty} -c_2(X)\cdot [x_i]\le 0$,
which in turn force $c_2(X)\cdot [x_i]$ tends to $0$ when $i\rightarrow\infty$. Therefore, we complete the proof of our claim.

Note that the vanishing of the first Chern class of $X$ implies that we can find a Ricci flat metric (still denoted $\omega$) in the K\"ahler class $[\omega]$ (cf. \cite{yau1978ricci} and \cite{yau2008survey}).
The vanishing $\int_X c_2(X)\wedge \omega^{n-2}=0$ implies that the full curvature tensor of $\omega$ is identically zero.
By the uniformisation theorem, the universal cover of $X$ is an affine space and $X$ is a quotient of a complex torus $T$ by a finite group $G$ acting freely on $T$.
\end{proof}
\begin{pro}\label{torus1}
Let $f:X\rightarrow X$ be an int-amplified endomorphism of a normal compact K\"ahler threefold with at worst canonical singularities.
Suppose that the canonical divisor $K_X$ is pseudo-effective. Then $X$ is a $Q$-torus.
\end{pro}
We refer the readers to \cite[Theorem 3.3]{nakayama2010polarized} and \cite[Theorem 5.2]{meng2017building} for the projective version while our present proof will be different. %Recall that for a compact K\"ahler space $X$, we have defined $\textup{N}_1(X)$ in Subsection \ref{section2.3}  to be the vector space of real closed currents of bidimension $(1,1)$ modulo the following equivalence: $T_1\equiv T_2$ if and only if $T_1(\eta)=T_2(\eta)$ for all real closed $(1,1)$-forms $\eta$ with local potentials.
%In the case of threefolds, we do not need $V_k$ is smooth for sufficiently large $k$. Also, we only need the terms for $k=1,2$  during the proof. Recall that $N_1(X)$ is the vector space of real closed currents of bidimension $(1,1)$ modulo the equivalence relation: $T_1\equiv T_2$ if and only if
 %$T_1(\eta)=T_2(\eta)$ for all real closed $(1,1)$-forms $\eta$ with local potentials. 
\begin{proof}
To prove Proposition \ref{torus1}, we resort to \cite[Theorem 1.1]{graf2019finite} by showing that $c_1(X)=0\in \textup{H}^2(X,\mathbb{R})$ and $\widetilde{c_2}(X)\cdot[\omega]=0$ for some K\"ahler class $[\omega]\in \textup{H}^2(X,\mathbb{R})$. Here $c_1(X)$ is the first Chern class and $\widetilde{c_2}(X)\in \textup{H}^2(X,\mathbb{R})^\vee$ is the  orbifold second Chern class of $X$ (cf.~\cite[Section 5]{graf2019finite}). 
Since $-K_X$ is pseudo-effective by Theorem \ref{thm1.3} and $K_X$ is pseudo-effective by assumption, $K_X$ is numerically trivial in the sense of currents, i.e., $c_1(X)_{\mathbb{R}}=0$. It follows from the ramification divisor formula $K_X=f^*K_X+R_f$ that $R_f$ is zero and thus $f$ is quasi-\'etale. Let $d:=\deg f>1$ and fix a K\"ahler class $[\omega]$ on $X$. 

Now we prove that $\widetilde{c_2}(X)\cdot[\omega]=0$. Since $f$ is quasi-\'etale, by \cite[Proposition 5.6]{graf2019finite}, we get 
$\widetilde{c_2}(X)\cdot f^*[\omega]=d(\widetilde{c_2}(X)\cdot [\omega])$.
Then with the equality divided by $d$, we have
\begin{align*}
\widetilde{c_2}(X)\cdot [\omega]=\widetilde{c_2}(X)\cdot\frac{f^*[\omega]}{d}=\widetilde{c_2}(X)\cdot\frac{(f^i)^*[\omega]}{d^i}=:\widetilde{c_2}(X)\cdot [x_i],
\end{align*}
where $\lim\limits_{i\rightarrow \infty}[x_i]\equiv_w0$ by Lemma \ref{degreelem}. Note that $\widetilde{c_2}(X)\in \textup{H}^2(X,\mathbb{R})^\vee$ and we can regard it as an element in the dual of the Bott-Chern cohomology space and hence $\widetilde{c_2}(X)\in \textup{N}_1(X)$ (cf.~\cite[Proposition 3.9]{horing2016minimal}). Besides, for any normal compact K\"ahler threefold $X$, the nef cone $\text{Nef}(X)$ and the cone $\overline{\textup{NA}}(X)$ generated by positive closed currents of bidimension $(1,1)$ are dual (cf.~\cite[Proposition 3.15]{horing2016minimal}). Therefore, 
$[\omega]^2$ is an interior point of $\overline{\textup{NA}}(X)$ because as a linear function on $\textup{H}^{1,1}_{\text{BC}}(X)$, $[\omega]^2$ is strictly positive on $\textup{Nef}(X)\backslash\{0\}$ by Hodge-Riemann theorem (also cf.~\cite[Th\'eor\`eme 3.1]{dinh2004groupes}). Therefore, there exists $A>0$ such that $A[\omega]^2\ge \widetilde{c_2}(X)$ and $A[\omega]^2\ge -\widetilde{c_2}(X)$ in the sense of currents. Hence, with the same argument as in the proof of Proposition \ref{torus2}, $\widetilde{c_2}(X)\cdot[\omega]=\lim\limits_{i\rightarrow \infty}\widetilde{c_2}(X)\cdot [x_i]=0.$

Then, applying the criterion in \cite[Theorem 1.1]{graf2019finite}, we see that $X$ is a $Q$-torus.
\end{proof}

\begin{lem}\label{torus3}
Let $f:X\rightarrow X$ be an int-amplified endomorphism of a normal compact K\"ahler surface with at worst klt singularities. Suppose that the canonical divisor $K_X$ is pseudo-effective. Then $X$ is a $Q$-torus.
\end{lem}

\begin{proof}
With the same arguments as in the previous proof (cf.~Proposition \ref{torus1}), $K_X$ is numerically trivial in the sense of currents, i.e., $c_1(X)_{\mathbb{R}}=0$. Therefore, it follows from the ramification divisor formula that $f$ is quasi-\'etale. Also, $\widetilde{c_2}(X)=0$ according to \cite[Proposition 5.6]{graf2019finite}. Hence,  $X$ is a $Q$-torus (cf.~\cite[Proposition 7.2]{graf2019finite}).
\end{proof}

\begin{lem}\label{torus4}
Let $f:X\rightarrow X$ be an int-amplified endomorphism of a normal compact K\"ahler non-uniruled surface such that $K_X$ is $\mathbb{Q}$-Cartier. Then $X$ is a $Q$-torus.
\end{lem}

\begin{proof}
By Theorem \ref{thm1.3}, $-K_X$ is pseudo-effective. 
Let $\pi:\widetilde{X}\rightarrow X$ be a projective resolution. 
Since $X$ and hence $\widetilde{X}$  are non-uniruled K\"ahler surfaces, $K_{\widetilde{X}}$ is pseudo-effective by the classification of surfaces. 
Further, the equality $\pi_*K_{\widetilde{X}}=K_X$ gives that $K_X$ is also pseudo-effective. Therefore, $K_X\equiv 0$.

Since $\dim X=2$, with the same proof as in \cite[Lemma 2.4]{hu2014ample}, $X$ has only canonical singularities.
Then it follows from Lemma \ref{torus3} that $X$ is a $Q$-torus.
\end{proof}

As a consequence of Lemma \ref{torus4}, we prove the next proposition.
\begin{pro}\label{torus5}
Let $f:X\rightarrow X$ be an int-amplified endomorphism of a normal $\mathbb{Q}$-factorial compact K\"ahler surface $X$ such that the canonical divisor $K_X$ is  pseudo-effective. Then $X$ is a $Q$-torus.
\end{pro}
\begin{proof}
Since $K_X$ is pseudo-effective, with the same reason as in previous proofs, $K_X\equiv 0$. Then, it follows from \cite[Theorem 1.2]{fujino2019minimal} that $K_X$ is semi-ample and  $K_X\sim_{\mathbb{Q}} 0$.
If $X$ is non-uniruled, this is the case in Lemma \ref{torus4}. Therefore, we consider the case of $X$ being uniruled.  
Suppose $X$ is a uniruled surface. Then, $X$ is bimeromorphic to a ruled surface which is algebraic. Thus, $X$ is Moishezon by \cite[Theorem 7.14]{ancona2006differential}. Taking the global index-one log-canonical cover  $\widetilde{X}\rightarrow X$ associated with $K_X\sim_\mathbb{Q} 0$, we have the following results (cf.~\cite[Lemma 3.2.7]{nakayama2008on} or \cite[Lemma 4.21]{nakayama2020-1920} and Lemma \ref{lem3.333}): (1) $\widetilde{X}$ is a normal Moishezon and K\"ahler surface; (2) $\widetilde{X}\rightarrow X$ is \'etale in codimension one; (3) $\widetilde{X}$ has only rational double points; (4) $K_{\widetilde{X}}\sim 0$; and
(5) $f$ lifts to an int-amplified endomorphism $\widetilde{f}$ of $\widetilde{X}$.
Therefore, it follows from \cite[Proposition 7.3.1]{nakayama2008on} or \cite[Proposition 5.2]{nakayama2020} that $\widetilde{X}$ and hence $X$ are $Q$-tori (cf.~\cite[Lemma 7.4]{graf2019finite}).
\end{proof}

\begin{proof}[Proof of Theorem \ref{thmtorus}]
Note that Theorem \ref{thmtorus} is a direct consequence of Propositions \ref{torus2}, \ref{torus1} and \ref{torus5}, which together  prove the cases:  manifolds, threefolds and surfaces respectively.	
\end{proof}

It is natural to consider whether Theorem \ref{thmtorus} can be generalized to arbitrary dimensional spaces with mild singularities. For the normal projective setting, the answer is affirmative (cf.~\cite[Theorem 5.2]{meng2017building}). In our present context, it is still an open question.
\setlength\parskip{0pt}

We end up this section by assuming the following hypothesis when discussing the minimal model program for the case when $K_X$ is pseudo-effective. With the assumption of \textbf{Hyp $\mathbf{A}$}, we refer the readers to \cite{horing2016minimal} for details.

\textbf{Hyp $\mathbf{A}$: $X$ is a  normal $\mathbb{Q}$-factorial compact K\"ahler threefold with at worst terminal singularities. Also, the canonical divisor $K_X$ is pseudo-effective.}
\begin{comment}
	\begin{enumerate}we begin with $X_0=X$.
\item If $K_{X_i}$ is nef, then we have finished;
\item If $K_{X_i}$ is not nef, there exists a $K_{X_i}$-negative extremal ray $R_i$ in the cone $\overline{NA}(X_i)$ (cf.~\cite[Theorem 1.2]{horing2016minimal});
\item By \cite[Theorem 1.3]{horing2016minimal}, the contraction $\varphi:X_i\rightarrow Y$ of $R_i$ exists;
\item If $R_i$ is divisorial, then we return back to Step (1) with $X_{i+1}=Y$; If $R$ is small, then by Mori's flip theorem (cf.~\cite[Theorem 0.4.1]{mori88flip}), the flip $\varphi^+:X_i^+\rightarrow Y$ exists, and we return back to Step 1 with $X_{i+1}=X^+$. 
\end{enumerate}
Observe that in both cases (3) and (4), $X_{i+1}$ satisfies \textbf{Hyp $\mathbf{A}$} again (cf.~\cite[Proposition 8.1]{horing2016minimal}). Also, this MMP will terminate after finitely many steps by the boundedness of Picard number and the difficulty. Here the difficulty $d(X)$ is defined by $d(X):=\#\{i|a_i<1\}$, where $K_Y=\mu^*K_X+\sum a_iE_i$ and $\mu:Y\rightarrow X$ is any resolution of singularities. 
\end{comment}
\section{$K_X$ not pseudo-effective case}\label{section7}
In this short section, let $X$ be a normal compact K\"ahler threefold with at worst canonical singularities. 
We refer the readers to \cite[Introduction]{Horing2015mori} for the arguments below.

By a theorem of Brunella \cite{brunella2006positivity} and \cite[Theorem 4.2]{horing2018bimero}, the canonical divisor $K_X$ is not pseudo-effective if and only if $X$ is covered by rational curves. 
We shall deal with the uniruled case in the following.  
Let $\pi:\widetilde{X}\rightarrow X$ be a projective resolution and  $\widetilde{X}\dashrightarrow Y$  the maximal rationally connected (MRC) fibration. 

If $\dim Y=0$, then $\widetilde{X}$ is algebraic (cf.~\cite{campana1981cor}).
We suppose that $\dim Y=1$. Then,  a general fibre $F$ of the MRC fibration is smooth and rationally connected of dimension two in $\widetilde{X}$.
If $\textup{H}^0({\widetilde{X}},\Omega_{\widetilde{X}}^2)\neq 0$, then there exists a global $2$-form $\eta$ on $\widetilde{X}$ and hence we get a non-zero holomorphic form on some rationally connected fibre $F$, a contradiction (cf.~\cite[Corollary 4.18]{debarre}).  
So, applying Hodge theory, we have $\textup{H}^2(\widetilde{X},\mathcal{O}_{\widetilde{X}})=0$ and also $\textup{H}^2(\widetilde{X},\mathbb{C})=\textup{H}^{1,1}(\widetilde{X},\mathbb{C})$. 
%By the exponential sequence, the map $\textup{H}^1(\widetilde{X},\mathcal{O}_{\widetilde{X}}^*)\rightarrow \textup{H}^2(\widetilde{X},\mathbb{Z})$ is surjective. In addition, the nonzero K\"ahler cone $\mathcal{K}(\widetilde{X})$ is open in $\textup{H}^{1,1}(\widetilde{X},\mathbb{R}):=\textup{H}^{1,1}(\widetilde{X},\mathbb{C})\cap \textup{H}^2(\widetilde{X},\mathbb{R})=\textup{H}^2(\widetilde{X},\mathbb{R})$ and hence it follows from the exponential sequence that there exists a positive line bundle on $\widetilde{X}$. 
By \cite[Proposition 4.10]{graf2018algebraic}, $\widetilde{X}$ is projective and thus $X$ is Moishezon (cf.~\cite[Theorem 7.14]{ancona2006differential}). 

In both cases, $X$ is Moishezon and also projective thanks to \cite[Theorem 1.6]{namikawa2002projectivity}. In conclusion, together with \cite[Theorem 1.10]{meng2017building}, we get the following lemma.
\begin{lem}\label{baselessthan1}
Let $X$ be a compact K\"ahler threefold with at worst canonical singularities.
Suppose $X$ is uniruled and the dimension of the base $Y$ of the MRC fibration $\pi:X\dashrightarrow  Y$ is at most $1$. 
Then $X$ is projective.
In particular, if $X$ further admits an int-amplified endomorphism, then there exists an $f$-equivariant MMP for $X$ after iteration.
\end{lem}

From now on, when discussing the MMP for normal compact K\"ahler uniruled threefolds, we always assume the following hypothesis (cf.~\cite{Horing2015mori}).

\textbf{Hpy $\mathbf{B}$:  $X$ is a normal $\mathbb{Q}$-factorial compact K\"ahler threefold with at worst terminal singularities. Also, the canonical divisor $K_X$ is not pseudo-effective and the base $Y$ of the MRC fibration $\pi:X\dashrightarrow Y$ has dimension two.}

%We assume \textbf{Hyp $\mathbf{B}$} and refer to \cite{Horing2015mori} for the minimal model program for the case when $K_X$ is not pseudo-effective.
%recall the minimal model program in %the present case. Readers may refer to 
 % for details. Let $X_0=X$ and there are several steps to run the MMP.
%\begin{enumerate}
%\item[\textup{(1)}] By \cite[Theorem 1.3]{Horing2015mori}, there exists a MMP: $X\dashrightarrow X'$ such that $K_{X'}+\omega'$ is nef for every normalised K\"ahler classes $\omega'$ (cf.~\cite[Definition 1.2]{Horing2015mori}) on $X'$;
%\item[\textup{(2)}] Applying \cite[Theorem 1.4]{Horing2015mori} to the variety $X'$, we  obtain a holomorphic fibration $\varphi:X'\rightarrow S'$ onto a surface $S'$ such that $-K_{X'}$ is $\varphi$-ample;
%\item[\textup{(3)}] Run the relative MMP of $X'$ over $S'$ by the relative version of the cone and contraction theorem;
%\item[\textup{(4)}] Since $K_{X'}$ is not pseudo-effective over $S'$, the outcome of the MMP $X'\dashrightarrow X''$ is a Mori fibre space $X''\rightarrow S''$ over $S'$, with $S''$ a normal compact K\"ahler surface dominating $S'$;
%\end{enumerate}
%This program will also terminate by the finiteness of Picard number and the difficulty. 

\section{Equivariant minimal model program: Proof of Theorem \ref{thmmmpthreefold}}\label{section8}

In this section, we always assume \textbf{Hyp $\mathbf{A}$} in Section \ref{section6} or \textbf{Hyp $\mathbf{B}$} in Section \ref{section7} when discussing the MMP for K\"ahler threefolds. Before proving Theorem \ref{thmmmpthreefold}, we refer to \cite[Definitions 3.19, 4.3 and 7.1]{horing2016minimal} for the basic notation and respective properties of contraction morphisms, divisorial rays and small rays. Indeed, all of them are similar to those in projective settings.
By \cite[Remark 7.2]{horing2016minimal}, if the extremal ray $R$ is small, then the corresponding Mori contraction contracts only finitely many curves. Moreover, \cite{horing2016minimal} shows  that if we assume \textbf{Hyp $\mathbf{A}$} in Section \ref{section6} or \textbf{Hyp $\mathbf{B}$} in Section \ref{section7}, then after a contraction of either divisorial or small rays, the end product is still a  normal $\mathbb{Q}$-factorial compact K\"ahler threefold with at worst terminal singularities. In addition, under the same assumption, \cite{Horing2015mori} shows that if $\pi:X\rightarrow Y$ is a Mori fibre space, then $Y$ is a normal $\mathbb{Q}$-factorial compact K\"ahler surface with at worst klt singularities.

In what follows, we slightly generalize \cite[Lemma 2.11]{zhang2010polarized} to the following lemma. We shall skip the proofs since they are the same as the proofs in \cite[Lemma 2.11]{zhang2010polarized} by applying Proposition \ref{lem6.1isomorphism} and \cite[Proposition 3.9]{horing2016minimal} to  the present case.

\begin{lem}\label{keylem1}
Let $X$ be a normal compact K\"ahler space with at worst rational singularities and $f:X\rightarrow X$ a  surjective endomorphism. Let $R_{\Gamma}:=\mathbb{R}_{\ge 0}[\Gamma]\subseteq \overline{\textup{NA}}(X)$ (which is the closed cone in $\textup{N}_1(X)$ generated by the classes of positive closed currents of bidimension $(1,1)$) be an extremal ray (not necessarily $K_X$-negative) with $\Gamma$ a positive closed current  of bidimension $(1,1)$. Then, we have:
\begin{enumerate}
\item[(1)] $R_{f_*\Gamma}$ is an extremal ray;
\item[(2)] If $C_1$ is a curve such that $[f(C_1)]\in R_{\Gamma}$, then $R_{[C_1]}$ is an extremal ray;
\item[(3)] Denote by $\Sigma_{\Gamma}$ the set of curves whose classes are in $R_{\Gamma}$. Then $f(\Sigma_{\Gamma})=\Sigma_{f_*\Gamma}$.
\item[(4)] $\Sigma_{\Gamma}=f^{-1}(\Sigma_{f_*\Gamma}):=\{C~\text{a curve}~|~f(C)\in \Sigma_{f_*\Gamma}\}$.
\end{enumerate}
\end{lem}

The result below is known in the projective case (cf.~ \cite[Lemma 8.1]{meng2017building}) and we show that it also holds in the analytic setting. We rewrite the proof in the analytic version by highlighting the differences and skipping the same parts. The main difference is that, here, we may not have integral ample divisors in the analytic setting. Nevertheless, since every analytic subvariety outside the singular locus determines an integral homology class, the idea of the previous proof is valid in the present case.
\begin{lem}\label{keylem2finite}
Let $f:X\rightarrow X$ be an int-amplified endomorphism of a normal compact K\"ahler space, which is of dimension $n\ge 1$ and has at worst rational singularities. Suppose $A\subseteq X$ is a closed subvariety with $f^{-i}f^i(A)=A$ for all $i\ge 0$. Then $M(A):=\{f^i(A)|i\in\mathbb{Z}\}$ is a finite set.
\end{lem}

\begin{proof}
Let $M_{\ge 0}(A):=\{f^i(A)|i\ge 0\}$.  With the same reason as in the first step of the proof in \cite[Lemma 8.1]{meng2017building},  we only need to prove $M_{\ge 0}(A)$ is finite.

% We first reduce to the proof of the finiteness of $M_{\ge 0}(A)$.  Indeed, if $M_{\ge 0}(A)$ is finite, there exists $0<r_1<r_2$ such that $f^{r_1}(A)=f^{r_2}(A)$. Then for any $i>0$ and sufficiently large $s\gg 1$,
%$f^{-i}(A)=f^{-i}f^{-sr_1}f^{sr_1}(A)=f^{sr_2-sr_1-i}(A)\in M_{\ge 0}(A)$. 
%Therefore,

We  show that $M_{\ge 0}(A)$ is a finite set by induction on the codimension of $A$ in $X$.  If $A=X$, there is nothing to prove. Suppose $k:=\dim A<\dim X$ and $d:=\deg f>1$. Denote by $\Sigma$ the union of $\text{Sing}(X)$, $f^{-1}(\text{Sing}(X))$ and the irreducible components in the ramification divisor $R_f$ of $f$. Let $A_i:=f^i(A)~(i\ge 0)$. Note that $\dim A_i=\dim A=k$.% since $f$ is a finite morphism.

\textbf{We claim that $A_i$ is contained in $\Sigma$ for infinitely many $i$.} If not, with $A$ replaced by some $A_{i_0}$, we may assume that $A_i$ is not contained in $\Sigma$ for all $i\ge 0$. Since $f^{-i-1}(A_{i+1})=A$ and $f^i$ is surjective, we have $f^{-1}(A_{i+1})=A_i$. Besides, $f$ is unramified on each $A_i$, and hence $\deg f|_{A_i}=\deg f$. Fix a K\"ahler form $\xi$ on $X$. Then, by projection formula,
$(\deg f)(\xi^k\cdot [A_{i+1}])=(\deg f|_{A_i})(\xi^k\cdot [A_{i+1}])=(f^*\xi)^k \cdot [A_i]$. Substituting this expression repeatedly, we have
\begin{equation}\label{eq1.1}
\xi^k\cdot [A_{i+1}]=\frac{((f^i)^*\xi)^k\cdot [A_1]}{d^i}.
\end{equation}

On the one hand,  all of these $A_i$ are integral subvarieties of dimension $k$ and not contained in the singular locus by assumption. Also, the homology classes of these $A_i\subseteq X$ form a discrete set (denoted by $\textup{NE}_k(X)$), which is contained in the integral lattice $H_{2k}(X,\mathbb{Z})/\textup{tors}$. 
  In addition, $\xi^k$ defines a norm on the closure of $\textup{NE}_k(X)$ %positive closed currents of bidimension $(k,k)$ generated by the currents of integration of $k$-dimensional subvarieties and 
 and $\xi^k$ is strictly positive on $\overline{\textup{NE}}_k(X)\backslash \{0\}$. \textbf{Thus, there exists a real constant $\mu>0$ such that $\xi^k\cdot [A_i]\ge \mu$ for every analytic subvariety $A_i$ of dimension $k$, which is not contained in the singular locus.} As a result, the left hand side of Equation (\ref{eq1.1}) is always no less than $\mu$. 

On the other hand, by Lemma \ref{lemvanish}, $\lim\limits_{i\rightarrow\infty}\xi^k\cdot [A_{i+1}]=0$. Therefore, taking the upper limit of Equation (\ref{eq1.1}), we get a contradiction and  prove the claim.
The remaining proof is the same as \cite[Lemma 8.1]{meng2017building} after replacing \cite[Lemma 2.3]{meng2017building} by Lemma \ref{productint}.
\end{proof}

%We pose the following question. So far, it has an affirmative answer for the case  when $X$ satisfies \textbf{Hyp $\mathbf{A}$} in Section \ref{section6} or \textbf{Hyp $\mathbf{B}$} in Section \ref{section7}.

%\begin{question}\label{ques8}
%\textup{Suppose $X$ is a  normal $\mathbb{Q}$-factorial compact K\"ahler space and $f:X\rightarrow X$ is a surjective endomorphism. Suppose further that $R_{\Gamma}=\mathbb{R}_{\ge 0}[\Gamma]$ is a $K_X$-negative extremal ray, where $\Gamma$ is a positive closed current of bidimension $(1,1)$. Whether there exists some suitable condition for $X$ so that the contraction morphism of $R_{\Gamma}$ exists?}
%\end{question}

Let $f:X\rightarrow X$ be a surjective endomorphism of a compact K\"ahler manifold $X$. Suppose $\pi:X\dashrightarrow Y$ is a meromorphic map to a complex variety. In \cite[Propositions 2.9 and 2.12]{horing2011non}, H\"{o}ring and Peternell gave us  nice criteria for $f$ to descend as a (meromorphic) endomorphism of $Y$ in terms of Albanese map, Iitaka fibration, rationally connected quotient or algebraic reduction.

\begin{comment}
To begin with this subsection, 
we first remark on the finiteness of a surjective endomorphism. 

\begin{remark}
\textup{If $f:X\rightarrow X$ is a surjective endomorphism of either a normal projective variety or a compact K\"ahler manifold $X$, then it follows from the projection formula or the Gysin morphism that $f$ is finite.}

\textup{Besides, suppose either \textbf{Hyp $\mathbf{A}$} in Section \ref{section6} or \textbf{Hyp $\mathbf{B}$} in Section \ref{section7}. Then, for each surjective endomorphism $f:X\rightarrow X$ of a normal compact K\"ahler threefold $X$ with at worst terminal singularities, the surjectivity of $f$ will imply the finiteness of $f$. Indeed, in this case, $X$ is smooth in codimension two and thus the singular locus of $f$ is discrete. Then, any positive dimensional subvariety cannot be contained in the singular locus and hence, the corresponding current of integration is nonzero by the numerical characterization of K\"ahler classes (cf.~\cite{Demailly2004numerical}). So, the finiteness of $f$ follows from \cite[Theorem 1]{li1995analytic}.}

\textup{However, in a more general setting, especially when $X$ is an arbitrary normal compact K\"ahler space, we could not exclude the case for the cohomology class of some subvariety to be zero. Thus, the finiteness of a surjective endomorphism cannot be determined.} 
\end{remark}
\end{comment}

Now, it is natural to ask whether there exists an equivariant descending for the contraction morphism of a $K_X$-negative extremal  ray. Actually, it has an affirmative answer under the assumption of \textbf{Hyp $\mathbf{A}$} in Section \ref{section6} or \textbf{Hyp $\mathbf{B}$} in Section \ref{section7} (cf.~Lemma \ref{lemequivariant}). The proof is the same as in \cite[Lemma 6.2]{meng2018building} for the projective setting after replacing \cite[Lemma 2.13]{meng2018building} by Lemma \ref{keylem1} and replacing \cite[Corollary 3.17]{kollar2008birational} by \cite[Proposition 8.1]{horing2016minimal}.

\begin{lem}\label{lemequivariant}
Suppose either \textbf{Hyp $\mathbf{A}$} in Section \ref{section6} or \textbf{Hyp $\mathbf{B}$} in Section \ref{section7}. 
Let $f$ be a surjective endomorphism of $X$ and $\pi:X\rightarrow Y$ a contraction  of a $K_X$-negative extremal ray $R_\Gamma:=\mathbb{R}_{\ge 0}[\Gamma]$ generated by a positive closed $(2,2)$-current $\Gamma$. 
Suppose further that $E\subseteq X$ is an analytic subvariety such that $\dim(\pi(E))<\dim E$ and $f^{-1}(E)=E$. 
Then, up to replacing $f$ by its power, $f(R_{\Gamma})=R_{\Gamma}$ (hence, for any curve $C$ such that $[C]\in R_{\Gamma}$, its image $[f(C)]$ still lies in $R_{\Gamma}$). Therefore, the contraction $\pi$ is $f$-equivariant, i.e., $f$ descends to a surjective endomorphism $g$ on $Y$.
\end{lem}

 Lemma \ref{lemequivariant} is also valid for any K\"ahler surface with klt singularities. Besides, Lemma \ref{lemequivariant} applies when  $\pi$ is a Mori fibre space (with $E=X$): in that case,  $\pi$ is $f^s$-equivariant for some $s>0$. 

\begin{pro}\label{prodivisorial}
Let $f:X\rightarrow X$ be an int-amplified endomorphism of a normal $\mathbb{Q}$-factorial compact K\"ahler threefold $X$ with at worst terminal singularities. If $\pi:X\rightarrow Y$ is a divisorial contraction of a $K_X$-negative extremal ray $R_\Gamma:=\mathbb{R}_{\ge 0}[\Gamma]$ generated by a positive closed $(2,2)$-current $\Gamma$. Then $f^s$ induces an int-amplified endomorphism of $Y$ for some $s>0$.
\end{pro}
\begin{proof}
Let $S:=\bigcup_{C\subseteq X,[C]\in R_{\Gamma}}C$.
By \cite[Lemma 7.5]{horing2016minimal}, $S$ is an irreducible Moishezon surface and $S\cdot C<0$ for all curves $C$ with $[C]\in R_\Gamma$. By Lemma \ref{keylem1}, $f^{-i}f^i(S)=S$ for all $i\ge 0$, since $f^i$ is surjective. Then, it follows from Lemma \ref{keylem2finite} that $M_{\ge 0}(S)$ is a finite set. So, we may assume $f^{-1}(S)=S$ with $f$ replaced by its power. Also,  $\pi$ is $f^s$-equivariant for some $s>0$ (cf.~Lemma \ref{lemequivariant}). By Lemma \ref {lem3.4}, the induced endomorphism  on $Y$ is int-amplified.
\end{proof}

Indeed, without Lemma \ref{keylem2finite}, 
the finiteness of the number of such $S$ in Proposition \ref{prodivisorial} also comes from the fact that all of these Moishezon surfaces consist of the integral part of the Zariski decomposition for $K_X$ (cf.~
\cite[Section 4.B. and the proof of Lemma 7.5]{horing2016minimal}).

The following result is the analytic version of \cite[Lemma 3.6]{zhang2010polarized} (cf.~\cite[(0.4.1)]{mori88flip} for the existence of flips of threefolds), the proof of which is the same since each contraction morphism of compact K\"ahler threefolds is projective ($-K_X$ is $\pi$-ample).

\begin{lem}\label{flipdescendingholomorphic}
Let $f:X\rightarrow X$ be a surjective endomorphism of a normal $\mathbb{Q}$-factorial compact K\"ahler threefold $X$ with at worst terminal singularities and $\sigma:X\dashrightarrow X'$ a flip with $\pi:X\rightarrow Y$ the corresponding flipping contraction of a $K_X$-negative extremal ray $R_\Gamma:=\mathbb{R}_{\ge 0}[\Gamma]$ generated by some  positive closed $(2,2)$-current $\Gamma$. Suppose that $R_{f_*\Gamma}=R_\Gamma$. Then the dominant meromorphic map $f^+:X^+\dashrightarrow X^+$ induced from $f$, is holomorphic. Both $f$ and $f^+$ descend to the same endomorphism of $Y$. 
\end{lem}

\begin{pro}\label{proflipdescending}
Let $f:X\rightarrow X$ be an int-amplified endomorphism of a normal $\mathbb{Q}$-factorial compact K\"ahler threefold $X$ with at worst terminal singularities and $\sigma:X\dashrightarrow X'$ a flip with $\pi:X\rightarrow Y$ the corresponding flipping contraction of a $K_X$-negative extremal ray $R_\Gamma:=\mathbb{R}_{\ge 0}[\Gamma]$ generated by some  positive closed current $\Gamma$ of bidimension $(1,1)$. Then there exists a $K_X$-flip of $\pi$, $\pi^+:X^+\rightarrow Y$ such that $\pi=\pi^+\circ\sigma$ and for some $s>0$, the commutativity is $f^s$-equivariant.
\end{pro}

\begin{proof}
The proof is the same as \cite[Lemma 6.5]{meng2018building} after replacing \cite[Lemmas 2.13, 6.1, 6.2 and  6.6]{meng2018building} by Lemmas \ref{keylem1}, \ref{keylem2finite}, \ref{lemequivariant} and \ref{flipdescendingholomorphic}. Also, Lemmas \ref{lem3.4} and \ref{lem3.333} show that the induced endomorphism $f^+$ in Lemma \ref{flipdescendingholomorphic} is int-amplified.
\end{proof}

The next Theorem follows from Lemmas \ref {lemequivariant}, \ref{flipdescendingholomorphic} and Propositions \ref{prodivisorial}, \ref{proflipdescending}.

\begin{thm}\label{thm8.13}
Let $f:X\rightarrow X$ be an int-amplified endomorphism of a $\mathbb{Q}$-factorial compact K\"ahler space $X$ with at worst terminal singularities. Let $\pi: X\dashrightarrow Y$ be a dominant rational map which is either a divisorial contraction or a Fano contraction or a flip induced by a $K_X$-negative extremal ray $R_{\Gamma}$. Then there exists an int-amplified endomorphism $g:Y\rightarrow Y$ such that $g\circ\pi=\pi\circ f$ with $f$ replaced by its power.
\end{thm}

\begin{proof}[Proof of Theorem \ref{thmmmpthreefold}]
 First, if $K_X$ is pseudo-effective, then $X=Y$ is a $Q$-torus by  Proposition \ref{torus1}. Next, we consider the case when $K_X$ is not pseudo-effective. If the base of the MRC fibration $X\dashrightarrow Z$ has dimension $\le 1$, then Theorem \ref{thmmmpthreefold} follows from Lemma \ref{baselessthan1} and \cite[Theorem 1.10]{meng2017building}, since $X$ is projective in this case.

Now, we assume \textbf{Hyp $\mathbf{B}$} in Section \ref{section7}. By \cite[Theorem 1.1]{Horing2015mori}, we may run the MMP for finitely many steps: $X=X_1\dashrightarrow \cdots\dashrightarrow X_j$ consisting of  contractions of extremal rays and flips, such that $X_j$ admits a Mori fibre space $X_j\rightarrow S=X_{j+1}$ onto a normal  $\mathbb{Q}$-factorial compact K\"ahler surface. Note that $S$ has at worst klt singularities (cf.~\cite[Theorem 1.1]{Horing2015mori}).
Note also that $S$ is bimeromorphic to the base of the MRC fibration of $X$; hence $S$ is non-uniruled and $K_S$ is pseudo-effective by considering a resolution of $S$.
By Proposition \ref{torus5}, $S$ is a $Q$-torus, which is our end product. 
 Replacing $f$ by its power, the above sequence is $f$-equivariant by Theorem \ref{thm8.13}.

By Proposition \ref{proholomorphic}, the composition $u_i: X_i\dashrightarrow Y$ is a morphism for each $i$. Also, for each small contraction $\pi_i: X_i\rightarrow Z_i$, we claim that \textbf{the meromorphic map $Z_i\dashrightarrow Y$ is also a morphism}. In the following,  let $\text{Exc}(\pi_i)$ denote the exceptional locus of $\pi_i$.

If the claim does not hold, then by the rigidity lemma, there exists a curve $C$ on $X_i$ with $[C]\in R$ such that $u_i(C)$ is a curve on $Y$. Here, $R$ is a $K_{X_i}$-negative small ray and $\pi_i$ is the small contraction of $R$. 
Recall that, under \textbf{Hyp $\mathbf{B}$}, $Y$ is a $Q$-torus of dimension $2$. 
By Lemmas \ref{keylem1} and \ref{keylem2finite}, with $f$ replaced by its power, $f^{-1}(\textup{Exc}(\pi_i))=\textup{Exc}(\pi_i)$. Since $\dim\text{Exc}(\pi_i)=1$,  with $f$ further replaced by its power, $u_i(C)$  is totally invariant under the induced int-amplified endomorphism $Y\rightarrow Y$ (cf.~\cite[Lemma 7.5]{cascini2019polarized}), a contradiction to Lemma \ref{lemma5.6}. Thus, $Z_i\rightarrow Y$ is also a morphism and the MMP is over $Y$.

By \cite[Lemmas 2.6 and 5.2]{meng2018building}, $X_i\rightarrow Y$ is equi-dimensional with every fibre being irreducible and rationally connected. Also, $K_{X_i}$ is not pseudo-effective for any $i<r$, otherwise $X_i$ is a $Q$-torus for some $i<r$ and the MMP ends. Therefore, we complete the proof of Theorem \ref{thmmmpthreefold}.
\end{proof}

\end{document}